\documentclass[10pt,a4paper]{amsart}
\usepackage{amsfonts,amsmath,amssymb}
\usepackage{hyperref}
\usepackage[all]{xy}

\usepackage{amssymb,amsthm,amsxtra}
\usepackage[usenames]{color}

\usepackage{amscd}
\usepackage{amsthm}
\usepackage{amsfonts}
\usepackage{amssymb}
\usepackage{mathrsfs}

\newtheorem*{theorem*}{Theorem}
\newtheorem{lemma}{Lemma}[subsection]

\newtheorem{remark}[lemma]{Remark}

\newtheorem{theorem}[lemma]{Theorem}
\newtheorem{definition}[lemma]{Definition}
\newtheorem{notation}[lemma]{Notation}
\newtheorem{property}[lemma]{Property}

\newtheorem*{conjecture*}{Conjecture}

\newtheorem{thm}[lemma]{Theorem}
\newtheorem{prop}[lemma]{Proposition}
\newtheorem{lem}[lemma]{Lemma}
\newtheorem{defn}[lemma]{Definition}
\newtheorem{notn}[lemma]{Notation}
\newtheorem{cor}[lemma]{Corollary}

\newtheorem{rem}[lemma]{Remark}

\newtheorem{introtheorem}{Theorem}

\newtheorem{introthm}[introtheorem]{Theorem}
\newtheorem{introlem}[introtheorem]{Lemma}

\baselineskip 18pt \textwidth 16cm  \theoremstyle{plain}

\newcommand{\tr}{\operatorname{Tr}}

\newcommand{\Hom}{\operatorname{Hom}}
\newcommand{\Iso}{\operatorname{Iso}}

\newcommand{\eps}{\varepsilon}

\renewcommand{\Im}{\operatorname{Im}}
\newcommand{\Ker}{\operatorname{Ker}}

\newcommand{\Z}{{\mathbb Z}}

\newcommand{\R}{{\mathbb R}}
\newcommand{\bR}{{\mathbb R}}
\newcommand{\C}{{\mathbb C}}

\newcommand{\Span}{{\operatorname{Span}}}

\newcommand{\End}{\operatorname{End}}

\newcommand{\bet}{{\beta}}

\newcommand{\alp}{{\alpha}}

\newcommand{\Fre}{{Fr\'{e}chet \,}}
\newcommand{\lcctv}{{locally convex complete topological vector }}

\newcommand{\Fou}{{\mathcal{F}}}

\newcommand{\g}{{\mathfrak{g}}}
\newcommand{\h}{{\mathfrak{h}}}
\newcommand{\z}{{\mathfrak{z}}}

\newcommand{\GL}{\operatorname{GL}}

\newcommand{\gl}{{\mathfrak{gl}}}

\newcommand{\Sym}{\operatorname{Sym}}

\newcommand{\Sc}{{\mathcal S}}
\newcommand{\Lie}{\operatorname{Lie}}

\newcommand{\sign}{\operatorname{sign}}

\newcommand{\n}{\mathfrak{n}}
\newcommand{\ou}{\overline{u}}
\newcommand{\hot}{\widehat{\otimes}}
\newcommand{\ctp}{\widehat{\otimes}}
\newcommand{\wO}{\widetilde{\Omega}}
\newcommand{\cM}{\mathcal{M}}
\newcommand{\gN}{\mathfrak{N}}
\newcommand{\Chi}{\mathfrak{X}}
\newcommand{\ot}{\leftarrow}
\newcommand{\fS}{\mathfrak{S}}
\newcommand{\fJ}{\mathfrak{J}}
\newcommand{\cJ}{\mathcal{J}}
\newcommand{\cT}{\mathcal{T}}

\begin{document}

\author{Avraham Aizenbud}
\address{Avraham Aizenbud, Faculty of Mathematics
and Computer Science, The Weizmann Institute of Science POB 26,
Rehovot 76100, ISRAEL.} \email{aizenr@yahoo.com}
\author{Dmitry Gourevitch}
\address{Dmitry Gourevitch,
School of Mathematics,
Institute for Advanced Study,
Einstein Drive, Princeton, NJ 08540 USA}
\email{dimagur@ias.edu}
\date{\today}
\title[Smooth Transfer (the Archimedean case)]{Smooth Transfer of Kloosterman Integrals\\(the Archimedean case)}
%\date{\today}
%
%2010 Mathematics Subject Classification:
% 20G20    	Linear algebraic groups over the reals, the complexes, the quaternions
% 22E30    	Analysis on real and complex Lie groups 
% 22E45    	Representations of Lie and linear algebraic groups over real fields: analytic methods

\begin{abstract}
We establish the existence of a transfer, which is compatible with Kloosterman integrals, between Schwartz functions on $\GL_n(\R)$ and
Schwartz functions on the variety of non-degenerate Hermitian forms.
Namely, we consider
an integral of a Schwartz function on $\GL_n(\R)$ along the orbits
of the two sided action of the groups of upper and lower unipotent
matrices twisted by a non-degenerate character. This gives a
smooth function on the torus. We prove that the space of all
functions obtained in such a way coincides with the space that is
constructed analogously when $\GL_n(\R)$ is replaced with the variety
of non-degenerate hermitian forms. We also obtain similar results
for $\gl_n(\R)$.

The non-Archimedean case is done in \cite{Jac1} and our proof
follows the same lines. However we have to face additional
%?? serious essential?
difficulties that appear only in the Archimedean case.
\end{abstract}

\maketitle

\tableofcontents

\section{Introduction}
%Let E/F be a quadratic extension of local non-Archimedean fields, let  : F� !
%{�1} be the corresponding quadratic character, and let   : F ! C� be a nontrivial
%character.

Let $N^n$ be the subgroup of upper triangular matrices in $GL_n$
with unit diagonal, and let $A^n$ be the group of invertible
diagonal matrices. We define a character $\theta : N^n(\R) \to
\C^{\times}$ by $$\theta(u) = \exp(i \sum_{k=1}^{n-1} u_{k,k+1}).$$

%The group $N^n(F) \times N^n(F)$ acts
%on $GL(m, F)$ by $g \mapsto n_1^t g n_2$
%, and the orbits that intersect $A^n(F)$ form a dense open subset
%.
Let $\Sc(GL_n(\R))$ be the space of Schwartz functions on
$GL_n(\R)$. We define a map $\Omega: \Sc(GL_n(\R)) \to
C^\infty(A^n)$ by
$$\Omega(\Phi)(a):=
\int_
{(u_1,u_2) \in N^n(\R)\times N^n(\R)}
\Phi(u_1^t a u_2) \theta(u_1 u_2) du_1 du_2.$$

Similarly, we let $S^n(\C)$ be the space of non-degenerate Hermitian matrices $n \times n$ .
We define a map $\Omega: \Sc(S^n(\C)) \to C^\infty(A^n)$ by
$$\Omega(\Psi)(a):=
\int_
{u \in N^n(\C)}
\Psi(\overline{u}^t a u) \theta(u \overline{u}) du.$$

We say that $\Phi \in \Sc(GL_n(\R))$ \textit{matches} $\Psi \in
\Sc(S^n(\C))$ if for every $a \in A^n(F)$ , we have
$$\Omega(\Phi)(a) =
 \gamma(a) \Omega(\Psi)(a),$$
where
$$ \gamma(a):= sign(a_1)sign(a_1a_2)...sign(a_1a_2, . . . , a_{n-1}) \text{ for } a = diag(a_1, a_2,..., a_n).$$

The main theorem of this paper is

\begin{introthm} \label{Main}
For every $\Phi \in \Sc(GL_n(\R))$ there is a matching $\Psi \in \Sc(S^n(\C))$, and conversely.
\end{introthm}
We also prove a similar theorem for $\gl_n$.

We also consider non-regular orbital integrals and prove that if two functions match then their non-regular orbital integrals are also equal (up to a suitable transfer factor). This implies in particular that regular orbital integrals are dense in all orbital integrals.

The non-Archimedean counterpart of this paper is done in \cite{Jac1,JacSing} and our proof
follows the same lines. However we have to face additional
difficulties that appear only in the Archimedean case.

For the motivation of this problem we refer the reader to
\cite{Jac1}.

In the case of $\GL(2,\R)$ Theorem \ref{Main} was proven in \cite{Jac2Real}, using different methods.

\subsection{A sketch of the proof}  $ $

First we show that the theorem for $\gl_n$ implies the theorem for
$\GL_n$. Then we prove the theorem for $\gl_n$ by induction. We
construct certain open sets $O_i \subset \gl_n(\R)$ (for their definition
see \S\S \ref{SubsecNot}) and use the \textit{intermediate
Kloosterman integrals} in order to describe $\Omega(\Sc(O_i))$ in
terms of $\Omega(\Sc(GL_i(\R)))$ and $\Omega(\Sc(gl_{n-i}(\R)))$.
This gives a smooth matching for $\Sc(O_i)$ by the induction
hypothesis. We denote $U:= \bigcup O_i$ and $Z:=\gl_n(\R)-U$ and obtain by
partition of unity smooth matching for $\Sc(U)$.

Then we use an important fact. Namely, if $\Phi$ matches $\Psi$
then the Fourier transform of $\Phi$ matches the Fourier transform
of $\Psi$ multiplied by a constant. This is proven in \cite{Jac1} in the non-Archimedean
case and the same proof holds in the Archimedean case. The proof
of this fact is based on an explicit formula for the Kloosterman
integral of the Fourier transform of $\Phi$ in terms of the
Kloosterman integral of $\Phi$ (see Theorem \ref{SecondIng}).

In order to complete the proof of the main theorem we prove the
following Key Lemma.
\begin{introlem} \label{IntroKeyLem}
Let $N^n \times N^n$ act on $\gl_{n}$ by $x \mapsto u_1^t x u_2$. Let $\chi$ denote the character of $N^n \times N^n$ defined by $\chi (u_1,u_2)=\theta(u_1u_2)$.

Then any function in $\Sc(\gl_{n}(\R))$ can be written as a sum $f+g+h$
s.t. $f$ is a Schwartz function on $U$, the Fourier transform
of $g$ is a Schwartz function on $U$ and $h$ is a
function that annihilates any $(N^n \times N^n, \chi)-$
equivariant distribution on $\gl_{n}(\R)$
and in particular
$\Omega(h)=0$.
\end{introlem}

\subsection{The spaces of functions considered}  $ $

Since the proof relies on Fourier transform, in the Archimedean case it would not be appropriate to consider the space of smooth compactly supported functions. Therefore we had to work with Schwartz functions. Theories of Schwartz functions were developed by various authors in various generalities. We chose for this problem the version developed in \cite{AG_Sc,AG_RhamShap} in the generality of Nash (i.e. smooth semi-algebraic) manifolds. In Appendix \ref{AppSchwartz} of the present paper we develop further the tools for working with Schwartz functions from \cite{AG_Sc,AG_RhamShap} and \cite[Appendix B]{AG_HC}, for the purposes of this paper.

\subsection{Difficulties that we encounter in the Archimedean case}  $ $

Roughly speaking, most of the additional difficulties in the
Archimedean case come from the fact that the space of Schwartz
functions in the Archimedean case is a topological vector space
unlike the space of Schwartz functions in the non-Archimedean case
which is just a vector space. Part of those difficulties are
technical and can be overcome using the theory of nuclear \Fre
spaces. However there are more essential difficulties in the Key
Lemma. Namely, in the non-Archimedean case the Key lemma is
equivalent to the following one

\begin{introlem} \label{NonArchKeyLem}
Any $(N^n \times N^n, \chi)$-equivariant distribution on
$\gl_{n}(\R)$ supported on $Z$, whose Fourier transform is also
supported on $Z$, vanishes.
\end{introlem}
Note that even this lemma is harder in the Archimedean since we
have to deal with transversal derivatives. However, this
difficulty is overcome using the fact that the transversal derivatives
are controlled by the action of stabilizer of a point on the
normal space to its orbit. This action is rather simple since it
is an algebraic action of a unipotent group.

The main difficulty, though, is that in the Archimedean case Lemma
\ref{NonArchKeyLem} in not equivalent to Lemma
\ref{IntroKeyLem} but only to the following weak version of it

\begin{introlem}
Any function in $\Sc(\gl_{n}(\R))$ can be approximated by a sum
$f+g+h$ s.t. $f$ is a Schwartz function on $U$, the Fourier
transform of $g$ is a Schwartz function on $U$ and $h$ is
a function that annihilates any $(N^n \times N^n, \chi)-$
equivariant distribution on $\gl_{n}(\R)$ and in particular
$\Omega(f)=0$.
\end{introlem}

We believe that the reason that the Key Lemma holds is a part of a
general phenomenon. To describe this phenomenon  note that a statement concerning equivariant
distributions can be reformulated to a statement concerning
closure of subspaces of Schwartz functions. The phenomenon is that
in many cases this statement holds without the
 need to consider the closure. We discuss two manifestations of this phenomenon in \S\S\S \ref{SubSubSecDualUn} and \ref{SubSubSecCoinv},
 and prove them in appendices \ref{AppCoinv} and \ref{AppDualUn}. The proofs there remind in their spirit the proof of the
 classical Borel Lemma.

%?? We should rewrite that part to use Borel lemma!

\subsection{Contents of the paper}$ $

In \S \ref{SecPrel} we fix notational conventions and list the basic facts on Schwartz functions and nuclear \Fre spaces that we will use.

In \S \ref{SecProofMain} we prove the main result.
In \S\S \ref{SubsecNot} we introduce the notation that we will use to discuss our problem, and reformulate the main result in this notation. In \S\S \ref{SubSecMainIng} we introduce the main ingredients of the proof: description of $\Omega(\Sc(O_i))$ using intermediate Kloosterman
integrals, inversion formula that connects Fourier transform to Kloosterman
integrals, and the Key lemma. In \S\S \ref{SubSecPfMain} we deduce the main result, Theorem \ref{Main}, from the main ingredients.

In \S \ref{sec_inv_for} we prove the inversion formula.

In \S \ref{SecPfKeyLem} we prove the Key lemma.

In \S \ref{sec:SingOrbitInt} we consider non-regular orbital integrals,  define matching for them and prove that if two functions match then their non-regular orbital integrals also match.

In appendix \ref{AppSchwartz} we give some complementary facts about Nash
manifolds and Schwartz functions on them and prove an analog of Dixmier - Malliavin Theorem and prove dual versions of special cases of  uncertainty principle and localization principle. Those versions are two manifestations of the phenomenon described above.

\subsection{Acknowledgments}$ $

We would like to thank \textbf{Erez Lapid} for posing this problem to us and for discussing it with us.

We thank \textbf{Joseph Bernstein} and \textbf{Gadi Kozma} for fruitful discussions.

We thank \textbf{Herve Jacquet} for encouraging us and for his useful remarks, and \textbf{Gerard Schiffmann} for sending us the paper \cite{Va}.

Both authors were partially supported by a BSF grant, a GIF grant, and an ISF Center
of excellency grant. A.A was also supported by ISF grant No. 583/09 and
D.G. by NSF grant DMS-0635607. Any opinions, findings and conclusions or recommendations expressed in this material are those of the authors and do not necessarily reflect the views of the National Science Foundation.

%? thank MPIM

\section{Preliminaries} \label{SecPrel}
\subsection{General notation} \label{GenNat}

\begin{itemize}
%\item In this paper all the algebraic varieties are defined over $F$.

%\item
%For an algebraic variety $X$ we
%denote by $X(F)$ the topological space or smooth manifold of $F$ points of $X$.

%\item We consider linear spaces as algebraic varieties and treat
%them in the same way.

%\item For an algebraic variety $X$ defined over $\R$ we denote by
%$X_{\C}$ the natural algebraic variety defined over $\R$  such
%that $X_{\C}(\R)=X(\C)$. Note that over $\C$, $X_{\C}$ is
%isomorphic to $X \times X$.

%\item For a group or an algebra $A$ we denote by ${\mathcal Z}(A)$
%the center of $A$.
%
\item All the algebraic varieties and algebraic groups we consider in this paper are real.

\item For a group $G$ acting on a set $X$ and a point $x \in X$ we denote by $Gx$ or by $G(x)$ the orbit of $x$, by $G_x$ the stabilizer of $x$ and by $X^G$ the set of $G$-fixed points in $X$.

\item For Lie groups $G$ or $H$ we will usually denote their Lie algebras by $\g$ and $\h$.

\item An action of a Lie algebra $\g$ on a (smooth, algebraic, etc) manifold $M$ is a Lie algebra homomorphism from $\g$ to the Lie algebra of vector fields on $M$.
Note that an action of a (Lie, algebraic, etc) group on $M$ defines an action of its Lie algebra on $M$.

\item For a Lie algebra $\g$ acting on $M$, an element $\alpha \in \g$ and a point $x \in M$ we denote by $\alpha(x) \in T_xM$ the value at point $x$ of the vector field corresponding to $\alpha$.
We denote by $\g x \subset T_xM$ or by $\g (x)$ the image of the map $\alpha \mapsto \alpha(x)$ and by $\g_x \subset \g$ its kernel.

\item For a Lie algebra (or an associative algebra) $\g$ acting on a vector space $V$ and a subspace  $L \subset
V$, we denote by $\g L \subset V$ the image of the action map $\g
\otimes L \to  V$.

\item For a representation $V$ of a Lie algebra $\g$ we denote by $V^{\g}$ the space of $\g$-invariants and by $V_{\g}:=V/\g V$ the space of $\g$-coinvariants.

\item  For manifolds  $L \subset M$ we
denote by $N_L^M:=(T_M|_L)/T_L $ the normal bundle to $L$ in $M$.

\item Denote by $CN_L^M:=(N_L^M)^*$ the conormal  bundle.

\item For a point
$y\in L$ we denote by $N_{L,y}^M$ the normal space to $L$ in $M$
at the point $y$ and by $CN_{L,y}^M$ the conormal space.

\item By bundle we always mean a vector bundle.

\item For a manifold $M$ we denote by $C^{\infty}(M)$ the space of infinitely differentiable functions on $M$, equipped with the standard topology.
\end{itemize}

\subsection{Schwartz functions on Nash manifolds} \label{Schwartz}$ $

We will require a theory of Schwartz functions on Nash manifolds
as developed e.g. in \cite{AG_Sc}. %This theory is developed for Nash manifolds.
Nash manifolds are smooth semi-algebraic manifolds but in the
present work, except of Appendix \ref{AppSchwartz}, only smooth real
algebraic manifolds are considered. Therefore the reader can
safely replace the word {\it Nash} by {\it smooth real algebraic}
in the body of the paper.

Schwartz functions are functions that decay, together with all
their derivatives, faster than any polynomial. On $\R^n$ it is the
usual notion of Schwartz function. For precise definitions of
those notions we refer the reader to \cite{AG_Sc}. We will use the
following notations.

\begin{notation}
Let $X$ be a Nash manifold. Denote by $\Sc(X)$ the \Fre space of
Schwartz functions on $X$.
%For any Nash vector bundle $E$ over $X$ we denote by $\Sc(X,E)$
%the space of Schwartz sections of $E$.
\end{notation}

We will need several properties of Schwartz functions from
\cite{AG_Sc}.

\begin{property}[\cite{AG_Sc}, Theorem 4.1.3] \label{pClass}  $\Sc(\R ^n)$ = Classical
Schwartz functions on $\R ^n$.
\end{property}

\begin{property}[\cite{AG_Sc}, Theorem 5.4.3] \label{pOpenSet}
Let $U \subset M$  be an open Nash submanifold, then
$$\Sc(U) \cong \{\phi \in \Sc(M)| \quad \phi \text{ is 0 on } M
\setminus U \text{ with all derivatives} \}.$$ In this paper we
will consider $\Sc(U)$ as a subspace of $\Sc(X)$.
\end{property}

\begin{property}[see \cite{AG_Sc}, \S 5]\label{pCosheaf}
Let $M$ be a Nash manifold. Let $M = \bigcup_{i=1}^n U_i$ be a finite
cover of $M$ by open Nash submanifolds. Then a function $f$ on $M$ is a Schwartz function if
and only if it can be written as $f= \sum \limits _{i=1}^n f_i$
where $f_i \in \Sc(U_i)$ (extended by zero to $M$).

Moreover, there exists a smooth partition of unity $1 =\sum
\limits _{i=1}^n \lambda_i$ such that for any Schwartz function $f
\in \Sc(M)$ the function $\lambda_i f$ is a Schwartz function on
$U_i$ (extended by zero to $M$).
\end{property}

%\begin{notation}
%Let $M$ be a Nash manifold. We denote by $D_M$ the Nash bundle of
%densities on $M$. It is the natural bundle whose smooth sections
%are smooth measures. For the precise definition see e.g.
%\cite{AG_Sc}.
%\end{notation}

\begin{property}[see \cite{AG_Sc}, \S 5] \label{Extension}
Let $Z \subset M$ be a Nash closed submanifold. Then restriction
maps  $\Sc(M)$ onto $\Sc(Z)$.
\end{property}

\begin{property}[\cite{AG_HC}, Theorem B.2.4] \label{NashSub}
Let $\phi:M \to N$ be a Nash submersion of Nash manifolds. Let $E$
be a Nash bundle over $N$. Fix Nash measures $\mu$ on $M$ and $\nu$ on $N$.

Then\\
(i) there exists a unique continuous linear map $\phi_*:\Sc(M) \to
\Sc(N)$ such that for any $f \in \Sc(N)$ and $g \in \Sc(M)$ we
have
$$\int_{x \in N}  f(x)\phi_*g(x) d\nu  = \int_{x \in
M} (f(\phi(x)))g(x)d\mu.$$ In
particular, we mean that both integrals converge. \\
(ii) If $\phi$ is surjective then $\phi_*$ is surjective.

In fact $$\phi_*g(x) = \int_{z \in \phi^{-1}(x)} g(z) d\rho$$ for an appropriate measure $\rho$.
\end{property}

We will need the following analog of Dixmier - Malliavin theorem.
\begin{property}\label{PropDixMal}
Let $\phi:M \to N$ be a Nash map of Nash manifolds. Then
multiplication defines an onto map $\Sc(M) \otimes \Sc(N)
\twoheadrightarrow \Sc(M)$.% This gives us an action of Schwartz
%functions on $M$ on Schwartz functions on $N$.
\end{property}
For proof see Theorem \ref{ThmDixMal}.

We will also need the following notion.
\begin{notn}
Let $\phi:M \to N$ be a Nash map of Nash manifolds. We call a function $f
\in C^{\infty}(M)$ \textit{Schwartz along the fibers of $\phi$}
if for any Schwartz function $g \in \Sc(N)$, we have $(g \circ
\phi)f \in \Sc(M).$

We denote the space of such functions by $\Sc^{\phi,N}(M)$. If there is no ambiguity we will sometimes denote it by $\Sc^{\phi}(M)$ or by $\Sc^{N}(M)$.
We
define the topology on $\Sc^{\phi}(M)$ using the following system
of semi-norms:
%to be the weakest topology such that
for any seminorms $\alpha$ on $\Sc(N)$ and $\beta$ on $\Sc(M)$ we
define $$\gN_{\beta}^{\alpha}(f):= \sup_{g \in \Sc(N)| \alpha(g)<1}
\beta(f (g \circ \phi)).$$
\end{notn}

We will use the following corollary of Property \ref{NashSub}.

\begin{cor} \label{NashSubCor}
Let $\phi:M \to N$ be a Nash map and $\psi:L \to M$ be a Nash submersion.
Fix Nash measures on $L$ and $M$.
Then there is a natural continuous linear map $\phi_*:\Sc^{N}(L) \to
\Sc^{N}(M).$
\end{cor}
%\begin{cor} \label{NashSubCor}
%Let $\phi:M \to N$ be a Nash map.
%Let $V \subset N$ be the set of regular values of $\phi$. Let $U:=\phi^{-1}(V)$.
%Fix Nash measures on $U$ and $V$.
%Then there is a natural continuous linear map $\phi_*:\Sc^{\phi}(U) \to
%\Sc^{Id}(V).$
%\end{cor}

\begin{rem} \label{rem:RelSch}
Let $\phi:M \to N$ be a Nash map of Nash manifolds. Let $V\subset
N$ be a dense open Nash submanifold. Let $U:=\phi^{-1}(V)$.
Suppose that $U$ is dense in $M$. Then we have embeddings
$$\Sc(M) \hookrightarrow \Sc^{\phi,N}(M) \hookrightarrow \Sc^{\phi,V}(U).$$
In this paper we will view
$\Sc(M)$ and $\Sc^{\phi,N}(M)$ as subspaces of $\Sc^{\phi,V}(U)$.
\end{rem}

%\begin{notn}
%For a Nash map $\phi:M \to N$ and a Nash bundle $E$ on $M$ we
%define $\Sc^{\phi}(M,E)$ similarly to $\Sc^{\phi}(M)$.
%\end{notn}

\subsubsection{Fourier transform}
%Let $V$ be a finite dimensional real vector space.
%algebraic representation of $G$ over $F$.
%From now till the end of the paper we fix an additive character
%$\kappa$ of $F$ given by $\kappa(x):=e^{2\pi i \re(x)}$.

%\begin{notation}
%Let $V$ be a finite dimensional real vector space. Let $\psi$ be a
%non-trivial additive character of $\R$. Let $\mu $ be a Haar
%measure on $V$. Let $f \in \Sc(V)$ be a function. We denote by
%$\widehat{f} \in \Sc(V^*)$ the Fourier transform of $f$ defined by
%$\mu$ and $\psi$.
%\end{notation}
\begin{notation}
Let $V$ be a finite dimensional real vector space.
Let $B$ be a
non-degenerate bilinear form on $V$ and $\psi$ be a non-trivial
additive character of $\R$. Then $B$ and $\psi$ define Fourier
transform with respect to the self-dual Haar measure on $V$. We
denote it by $\Fou_{B,\psi}: \Sc(V) \to \Sc(V)$.
%For any Nash manifold $M$ we also denote by $\Fou_B:\Sc^*(M \times
%V) \to \Sc^*(M \times V)$ the fiberwise Fourier transform.
If there is no ambiguity, we will omit $B$ and $\psi$.
We will also denote by $\Fou_{B,\psi}^*: \Sc^*(V) \to \Sc^*(V)$ the dual map.
\end{notation}

We will use the following trivial observation.

\begin{lemma}
Let $V$ be a finite dimensional real vector space. Let a Nash
group $G$ act linearly on $V$. Let $B$ be a $G$-invariant
non-degenerate symmetric bilinear form on $V$. Let $\psi$ be a
non-trivial additive character of $\R$.
%Let $M$ be a Nash manifold
%with an action of $G$.
Then $\Fou_{B,\psi}$ commutes with the
action of $G$.
\end{lemma}

\subsubsection{Dual uncertainty principle} \label{SubSubSecDualUn}

\begin{theorem} \label{DualUncerPrel}
Let $V$ be a finite dimensional real vector space.
Let $B$ be a
non-degenerate bilinear form on $V$ and $\psi$ be a non-trivial
additive character of $\R$. Let $L \subset V$ be a
subspace. Suppose that $L^\bot \nsubseteqq L$.
Then $$\Sc(V-L) + \Fou(\Sc(V-L))=\Sc(V).$$
\end{theorem}
For proof see Appendix \ref{AppDualUn}.

\begin{rem}
It is much easier to prove that $$\overline{\Sc(V-L) +
\Fou(\Sc(V-L))}=\Sc(V)$$ since this is equivalent to the
fact that there are no distributions on $V$ supported in $L$ with
Fourier transform supported in $L$.
\end{rem}

%\begin{theorem} \label{DualUncerPrel}
%Let $V$ be a linear space, $L \subset V$ and $L' \subset V^*$ be
% subspaces. Suppose that $(L')_\bot \nsubseteqq L$.
% %Let $E:= \{f \in
%% \Sc(V)| \, f \text{ vanishes with all its derivatives on } L\}$ and $E':= \{f \in
%% \Sc(V^*)| \, f \text{ vanishes with all its derivatives on } L'\}$.
%Then $$\Sc(V-L) + \widehat{\Sc(V^*-L')}=\Sc(V).$$
%\end{theorem}
%For proof see Appendix \ref{AppDualUn}.%%%
%
%\begin{rem}
%It is much easier to prove that $\overline{\Sc(V-L) +
%\widehat{\Sc(V^*-L')}}=\Sc(V)$ since this is equivalent to the
%fact that there are no distributions on $V$ supported in $L$ with
%Fourier transform supported in $L'$.
%\end{rem}%

\subsubsection{Coinvariants in Schwartz functions}
\label{SubSubSecCoinv}
$ $\\
%{To Right a stronger vertion of the Theorem}
\begin{thm} \label{thm:NegCoinvPrel}
Let a connected algebraic group $G$ act on a real algebraic manifold $X$.
Let $Z$ be a $G$-invariant Zariski closed subset of $X$. Let $\g$ be the Lie algebra of $G$.
Let $\chi$ be a
unitary character of $G$.
Suppose that for any $\z \in Z$ and $k \in \Z_{\geq 0}$ we have
$$(\chi \otimes \Sym^k(CN_{z,Gz}^X) \otimes ((\Delta_G)|_{G_z}/\Delta_{G_z}))_{\g_z} =0.$$
Then $$\Sc(X) =\Sc(X-Z)+\g(\Sc(X)\otimes \chi).$$
\end{thm}
For proof see Appendix \ref{AppCoinv}.

\begin{cor} \label{cor:NegCoinvPrel}
Let a unipotent group $G$ act on a real algebraic manifold $X$.
Let $\chi$ be a
unitary character of $G$.

%Suppose that there exists a real
%algebraic group $G'\vartriangleright G$ which acts on $X$ extending the action of $G$.
Let $Z\subset X$ be a Zariski closed $G$-invariant subset.
%Suppose that $Z$ lies in a union of a finite number of $G'$-orbits.
Suppose also that for any point $z\in Z$ the restriction $\chi|_{G_z}$ is non-trivial.
Then $$\Sc(X)\otimes \chi =\Sc(X-Z)\otimes \chi+\g(\Sc(X) \otimes \chi),$$ where $\g$ is the Lie algebra of
$G$.
\end{cor}

\begin{proof}
The action of $G_z$ on $\Sym^k(CN_{z,Gz}^X) \otimes ((\Delta_G)|_{G_z}/\Delta_{G_z})$ is algebraic and hence if $G$ is unipotent this action is unipotent and therefore if $(\chi)_{\g_z} =0$ then $$(\chi \otimes \Sym^k( CN_{z,Gz}^X) \otimes
((\Delta_G)|_{G_z}/\Delta_{G_z}))_{\g_z} =0.$$
\end{proof}

%\begin{thm} \label{NegCoinvPrel}
%%Let a real algebraic group $G$ act on a real algebraic manifold
%%$X$. Let $\chi$ be a unitary character of $G$.
%%Let $Z\subset X$ be a closed $G$-invariant subset. Suppose that for any point $z
%%\in Z$ the restriction $\psi|_{G_z}$ is non-trivial.
%%Suppose that there exists a real algebraic group
%%$G'\vartriangleright G$ which acts on $Z$ with finite number of
%%orbits extending the action of $G$.
%Let a unipotent group $G$ act on a real algebraic manifold $X$.
%Let $\chi$ be a
%unitary character of $G$. Suppose that there exists a real
%algebraic group $G'\vartriangleright G$ which acts on $X$ extending the action of $G$.
%Let $Z\subset X$ be a Zariski closed $G$-invariant subset. Suppose that $Z$ lies in a union of a finite number of $G'$-orbits.
%Suppose also that for any point $z\in Z$ the restriction $\chi|_{G_z}$ is non-trivial.
%Then $$\Sc(X)\otimes \chi =\Sc(X-Z)\otimes \chi+\g(\Sc(X) \otimes \chi),$$ where $\g$ is the Lie algebra of
%$G$.
%\end{thm}

%\begin{rem}
%Note that the statement that $\Sc(X)\otimes \chi =\overline{\Sc(X-Z)\otimes \chi+\g(\Sc(X)\otimes \chi)}$ is
%equivalent to the statement that any $G$-invariant distribution on $X$
%which is supported on $Z$ vanishes, which follows from the results of \cite{Va}.
%\end{rem}
\begin{rem}
Note that the statement that $\Sc(X)\otimes \chi =\overline{\Sc(X-Z)\otimes \chi+\g(\Sc(X)\otimes \chi)}$ is
equivalent to the statement that any $G$-invariant distribution on $X$
which is supported on $Z$ vanishes, which is a generalization of a result from \cite{Va}.
\end{rem}
%??Rami

\subsection{Nuclear \Fre spaces} \label{NucFre}$ $

A good exposition on nuclear \Fre spaces can be found in Appendix
A of \cite{CHM}.

We will need the following well-known facts from the theory of
nuclear \Fre spaces.

\begin{prop}[see e.g. \cite{CHM}, Appendix A]\label{NFSSubQu}
\item Let $V$  be a nuclear \Fre space and $W$ be a closed subspace. Then both $W$ and $V/W$ are nuclear \Fre
spaces.
\end{prop}

%??Rami
\begin{prop} [see e.g.  \cite{CHM}, Appendix A] \label{prop:ctp_ext}
\item Let $0 \to V \to W \to U  \to 0$  be an exact sequence of nuclear \Fre spaces. Suppose that the embedding $V \to W$ is closed. Let   $L$  be a nuclear \Fre space. Then the sequence  $0 \to V \ctp L \to W  \ctp L \to U \ctp L \to 0$ is exact and the embedding  $V \ctp L \to W \ctp L$ is closed.
\end{prop}
\begin{cor} \label{OntoTensorOnto}
\item Let $V \to W$  be onto map between nuclear \Fre spaces and $L$  be a nuclear \Fre space. Then the map $V \ctp L\to W \ctp L $ is onto.
\end{cor}
%??Rami

\begin{cor}\label{CorOnto}
Let $\phi_i: V_i \to W_i$ $i=1,2$ be onto maps between nuclear \Fre
spaces. Then the map $\phi_1 \ctp \phi_2: V_1 \ctp V_2 \to W_1
\ctp W_2$ is onto.
\end{cor}

\begin{prop}[see e.g. \cite{AG_RhamShap}, Corollary 2.6.2]
\item Let $M$ be a Nash manifold. Then
$\Sc(M)$ is a nuclear \Fre space.
\end{prop}

\begin{prop}[see e.g. \cite{AG_RhamShap}, Corollary 2.6.3]
\label{SchKer}
\item Let $M_i$, $i=1,2$ be Nash manifolds %and $E_i$ be Nash bundles over $M_i$.
Then $$\Sc(M_1 \times M_2)=\Sc(M_1) \ctp \Sc(M_2).$$
%Then $$\Sc(M_1 \times M_2,E_1 \boxtimes
%E_2)=\Sc(M_1,E_1) \ctp \Sc(M_2,E_2),$$ where $E_1 \boxtimes E_2$
%denotes the exterior product.
\end{prop}

\begin{defn}\label{Rotem}
By a subspace of a topological vector space $V$ we mean a linear subspace $L \subset V$ equipped with a topology such that the embedding $L \subset V$ is continuous.

Note that by Banach open map theorem if $L$ and $V$ are
nuclear \Fre spaces and $L$ is closed in $V$ then the topology of
$L$ is the induced topology from $V$.

By an image of a continuous linear map between topological vector
spaces we mean the image equipped with the quotient topology.
Similarly for a  continuous linear map between topological vector
spaces $\phi: V_1 \to V_2$ and a subspace $L \subset V_1$ we
the image $\phi(L)$  to be equipped
with the quotient topology.

Similarly a \textbf{sum} of two subspaces will be considered with
the quotient topology of the direct sum.
\end{defn}

\begin{rem}\label{SubspaceRem}
Note that by Proposition \ref{NFSSubQu}, sum of nuclear \Fre spaces and
image of a nuclear \Fre space are nuclear \Fre spaces.

Note also the operations of taking sum of subspaces and image of subspace commute.

%??Rami
Finally note that if $L$ and $L'$ are two nuclear \Fre subspaces of a complete locally convex topological vector space $V$ which coincide as linear subspaces then they are the same. Indeed, by Banach open map theorem they are both the same as $L+L'$.
%??Rami
\end{rem}

\begin{notn}
Let $V_i$, $i=1,2$ be \lcctv spaces. Let $L_i \subset V_i$ be
subspaces. We denote by $\cM_{L_1,L_2}^{V_1,V_2}:L_1 \ctp L_2 \to
V_1 \ctp V_2$ the natural map.
\end{notn}

From Corollary \ref{CorOnto} we obtain the following corollary.
\begin{cor}\label{ImTensorIm}
Let $V_i$, $i=1,2$ be \lcctv spaces. Let $L_i$, $i=1,2$ be nuclear
\Fre spaces. Let $\phi_i:L_i \to V_i$ be continuous linear maps.
Then $$\Im(\phi_1 \ctp \phi_2) =
\Im(\cM_{\Im(\phi_1),\Im(\phi_2)}^{V_1,V_2}).$$
\end{cor}

\begin{notn}
Let $M_i$, $i=1,2$ be smooth manifolds. We denote by
$\cM_{M_1,M_2}:C^\infty(M_1) \ctp C^\infty(M_2) \to C^\infty(M_1
\times M_2)$ the product map. For two subspaces $L_i \subset
C^\infty(M_i)$ we denote by $\cM_{L_1,L_2}:L_1 \ctp L_2 \to
C^\infty(M_1 \times M_2)$ the composition $\cM_{M_1,M_2} \circ
\cM_{L_1,L_2}^{C^{\infty}(M_1),C^{\infty}(M_2)}.$
%
% of the last map with the
%map $L_1 \ctp L_2 \to C^\infty(M_1) \ctp C^\infty(M_2)$.
%Let $M_i$, $i=1,2$ be smooth manifolds and $E_i$ be vector bundles
%over $M_i$. we denote by $\cM_{M_1,M_2}:C^\infty(M_1,E_1) \ctp
%C^\infty(M_2,E_2) \to C^\infty(M_1 \times M_2,E_1 \boxtimes E_2)$
%the product map, for two subspaces $L_i \subset C^\infty(M_i,E_i)$
%we denote by $\cM_{L_1,L_2}:L_1 \ctp L_2 \to C^\infty(M_1 \times
%M_2,E_1 \boxtimes E_2)$ the composition of the last map with the
%map $L_1 \ctp L_2 \to C^\infty(M_1,E_1) \ctp C^\infty(M_2,E_2)$
\end{notn}

\section{Proof of the main result} \label{SecProofMain}

\subsection{Notation}\label{SubsecNot}$ $

In this paper we let $D$ be a semi-simple 2-dimensional algebra
over $\R$, i.e. $D = \C$ or $D= \R \oplus \R$.
Let $a \mapsto \overline{a}$ denote the non-trivial involution of $D$, i.e. complex conjugate or swap.
Let $n$ be a
natural number. Let
%$\psi:F \to \C^{\times}$
$\psi:\R \to \C^{\times}$
be a nontrivial
character. The following notation will be used throughout the body of the
paper. In case when there is no ambiguity we will omit from the
notations the $n$, the $D$ and the $\psi$.

%We denote
\begin{itemize}
\item Denote by $H^n(D)$ the space of hermitian matrices of size $n$.
\item Denote $S^n(D):=H(D) \cap GL_n(D)$.
\item Denote by $\Delta_i^n:H \to \R$ the  main $i$-minor.
\item Let $O_i^n(D) \subset H$ be the subset of matrices with
$\Delta_i \neq 0$.
\item Let $U^n(D):= \bigcup_{i=1}^{n-1} O_i$ and $Z^n(D):=H - U$.
\item Let $N^n(D)<GL_n(D)$ be the subgroup consisting  of upper
unipotent matrices.
\item Let $\n^n(D)$ denote the Lie algebra of $N^n$.
\item We define a character $\chi_{\psi} : N \to
\C^{\times}$ by $\chi_{\psi}(x) := \psi(\sum_{i=1}^{n-1}
(x_{i,i+1} + \overline{ x_{i,i+1}}))$.
\item  Let the group $N$  act on $H$ by $x \mapsto
\ou^t x u$.
\item Fix a symmetric $\R$-bilinear form $B^n_D$ on $H$ by $B(x,y):=
\tr_{\R}(x w yw)$, where $w:=w_{n}$ is the longest element in the Weyl
group of $GL_n$.
\item Denote by $A^n < GL_n(\R)$ the subgroup of diagonal
matrices. We will also view $A^n$ as a subset of $S^n(D)$.

\item  %Let $\det:S^n(D) \to \R^{\times}$ be the determinant map.
Define
$\Omega_{D}^{n,\psi}: \Sc^{\det,\R^{\times}}(S^n(D)) \to C^{\infty}(A^n)$
by
$$ \Omega_{D}^{n,\psi}(\Psi)(a):= \int_{N} \Psi(\ou^t a u) \chi(u)du.$$
Here, $du$ is the standard Haar measure on $N$.

For proof that the integral converges absolutely, depends smoothly on $a$ and defines a continuous map
$\Sc^{\det}(S^n(D)) \to C^{\infty}(A^n)$ see Proposition \ref{OmegaWell}.
By Remark \ref{rem:RelSch} $\Omega_{D}^{n,\psi}$ defines in particular a  continuous map $\Sc(H^n(D)) \to C^{\infty}(A^n).$

%\item Denote by $\eps_n^n$ the $m \times n$ matrix whose first row
%is $(0,...,0,1)$ and the others are zero.

\item Denote by $N_i^n(D) < N^n(D)$ the subgroup defined by $$N_i^n(D) := \left \{ \begin{pmatrix}
  Id_{i} & * \\
  0 & Id_{n-i}
\end{pmatrix}
\right \}.$$

\item %Consider $\Delta_i$ as a map $O_i \to \R^{\times}$.
Define $\Omega_{D,i}^{n,\psi}: \Sc^{\Delta_i}(O_i^n) \to \Sc^{\Delta_i,\R^{\times}}(S^i \times
H^{n-i})$, where $S^i \times H^{n-i}$ is considered as a subspace
of $H_n$, in the following way
$$ \Omega_{D,i}^{n,\psi}(\Psi)(a):= \int_{N_i^n} \Psi(\ou^t a u) \chi(u) du .$$
Here, $du$ is the standard Haar measure on $N_i^n$.

For proof that the integral converges absolutely, depends smoothly on $a$ and defines a continuous map
$\Sc^{\Delta_i}(O_i^n) \to \Sc^{\Delta_i}(S^i \times
H^{n-i})$ see Proposition \ref{OmegaWell}.

\item Define a character $\eta_D :\R^{\times} \to \{\pm 1\}$ by
$\eta_D = 1$ if $D = \R \oplus \R$ and $\eta_D = \sign$ if $D =
\C$.

\item Define $\sigma:H^n(D) \to \R$ by $\sigma (x) :=
\prod_{i=1}^{n-1} \Delta^n_i(x)$.

\item Define $\wO_D^{n,\psi}: \Sc^{\det,\R^{\times}}(S^n) \to C^{\infty}(A^n)$
by
$$ \wO_D^n(\Psi)(a):= \eta(\sigma(a)) |\sigma(a)|\Omega(\Psi)(a)$$

\item Define $\wO_{D,i}^{n,\psi}: \Sc^{\Delta_i,\R^{\times}}(O_i^n) \to \Sc^{\Delta_i,\R^{\times}}(S^i \times
H^{n-i})$, in the following way
%$$ \wO_{D,i}^{n,\psi}(\Psi)(a):= \eta(\Delta_i(a)) |\Delta_i(a)|  \Omega_{i}^{n} % .$$
$$ \wO_{D,i}^{n,\psi}(\Psi)(a):= \eta(\Delta_i(a))^{n-i}|\Delta_i(a)|^{n-i}  \Omega_{i}^{n}  .$$
%Here, $du$ is the standard Haar measure on $N_i^n$.
%\item Define $\Omega_{D,-}^n: \Sc^{\det}(S^n(D)) \to C^{\infty}(A^n)$
%by
%$$ \Omega_D^n(\Psi)(a):= \int_{N} \Psi(\ou^t a u) \overline{\theta}(u\ou)du.$$
%Analogously we define $\widetilde{\Omega_D}_{-}^n$. %,$\Omega_{i,-}^n(D)$ and $\widetilde{\Omega}_{i,-}^n(D)$.

\item We define $\Omega_D^{n_1,...,n_k,\psi}:\Sc^{\det \times ... \times \det}(S^{n_1}(D) \times ... \times S^{n_k}(D)) \to C^{\infty}(A^{n_1} \times ...
\times A^{n_k})$ in a similar way to $\Omega_D^{n,\psi}$.
Analogously we define $\wO_D^{n_1,...,n_k,\psi}$.
%$\Omega_{D,-}^{n_1,...,n_k}$ and
%$\widetilde{\Omega}_{D,-}^{n_1,...,n_k}$.
\end{itemize}

\begin{prop} \label{OmegaWell}$ $\\
(i) The integral $\Omega_{D}^{n,\psi}$ converges absolutely and defines a continuous map
$\Sc^{\det}(S^n(D)) \to C^{\infty}(A^n)$.\\
(ii) The integral $\Omega_{D,i}^{n,\psi}$ converges absolutely and defines a continuous map
$\Sc^{\Delta_i}(O_i^n) \to \Sc^{\Delta_i}(S^i \times
H^{n-i})$.
\end{prop}

\begin{proof}$ $\\
%(i) Consider the action map $N \times A \to S$. Note that it is an open embedding and denote its image by $V$.
%We consider the standard Haar measures on $A$ and $N$, and their multiplication on $V$.
%Consider the projections: $\alpha:V \to N$ and $\beta:V \to A$.  Define $\Omega': \Sc^{\beta,A}(V) \to \Sc^{Id}(A)$ by $\Omega'(f)=\beta_*(\Upsilon f)$.
%Now, $\Omega$ is given by the following composition
%$$\Sc^{\det,\R^{\times}}(S) \subset \Sc^{\det,\R^{\times}}(V) \overset{\Omega'}{\to} \Sc^{\det,\R^{\times}}(A)\subset %C^{\infty}(A).$$
(i) Consider the map $\beta:H \to \R^n$ defined by $\beta=(\Delta_1,...,\Delta_n)$. Consider $A$ to be embedded in $\R^n$ by $(t_1,...t_n) \mapsto (t_1,t_1t_2,...,t_1t_2...t_n)$. Let $V:=\beta^{-1}(A) \subset H.$ Let $p_n:\R^n\to \R$ denote the projection on the last coordinate.
Note that the action map defines an isomorphism  $N \times A \to V$. Let $\alpha:V \to N$ denote the projection.
Let $\Chi \in \Sc^{Id}(V)$ be defined by $\Chi(v):=\chi(\alpha(v))$.
Define $\Omega': \Sc^{\beta,A}(V) \to \Sc^{Id}(A)$ by $\Omega'(f):=\beta_*(\Chi f)$.
Now, $\Omega$ is given by the following composition
$$\Sc^{\det,\R^{\times}}(S)
\subset
 \Sc^{\beta,p_n^{-1}(\R^{\times})}(S)
\subset \Sc^{\beta,A}(V) \overset{\Omega'}{\to}
 \Sc^{Id}(A) \subset C^{\infty}(A).$$
(ii) Consider $S^{i} \times H^{n-i}$ as a subset in $H^n$. Denote it by $B$. Consider the action map $N_i \times B \to H$. Note that it is an open embedding and its image is $O_i$.
We consider the standard Haar measures on $B$ and $N_i$, and their multiplication on $O_i$.
Consider the projections: $\alpha_i:O_i \to N_i$ and $\beta_i:O_i \to B$. Let $\Chi_i \in \Sc^{Id}(O_i)$ be defined by $\Chi_i(v):=\chi(\alpha_i(v))$. Consider $(\beta_i)_*:\Sc^{\Delta_i,\R^{\times}}(O_i) \to \Sc^{\Delta_i,\R^{\times}}(B)$.
Now, $\Omega_i(f)=(\beta_i)_*(\Chi_i f)$.
%$\Omega_i': \Sc^{\Delta_i,\R^{\times}}(O_i) \to \Sc^{\Delta_i,\R^{\times}}(B)$ by
%Now, $\Omega$ is given by the following composition
%$$\Sc^{\det,\R^{\times}}(S) \subset \Sc^{\det,\R^{\times}}(V) \overset{\Omega'}{\to} \Sc^{\det,\R^{\times}}(A)\subset C^{\infty}(A).$$
%
%Consider the map $\phi:H \to \R^n$ defined by $\phi(h):=(\Delta_1(h),...,\Delta_n(h)$.
%Note that $\phi|_{A^n}$ is an open embedding and we will consider $A^n$ to be embedded in $\R^n$ through $\phi$.
%by Corollary{NashSubCor}.
\end{proof}

The main theorem (Theorem \ref{Main}) can be reformulated now in the
following way:

\begin{thm} \label{ReformMain} $ $\\
(i) $ \widetilde{\Omega}_{\R \oplus \R}(\Sc(H(\R \oplus \R))) =
\widetilde{\Omega}_{\C}(\Sc(H( \C))).$\\ %Moreover, for any open set
%$U \subset \Sc(H(\R \oplus \R))$, the set
%
%$\widetilde{\Omega}_{\C}^{-1}(\widetilde{\Omega}_{\R \oplus
%\R}(U))$ is open in $\Sc(H( \C))$ and for any open set $W \subset
%\Sc(H(\C))$, the set $\widetilde{\Omega}_{\R \oplus
%\R}^{-1}(\widetilde{\Omega}_{\C}(W))$ is open in $\Sc(H( \R \oplus
%\R))$.\\
%Moreover, there exist continuous linear maps $\phi: \Sc(H(\R
%\oplus \R)) \to \Sc(H( \C))$ and $\psi:\Sc(H(\C)) \to \Sc(H( \R
%\oplus \R))$ such that $\widetilde{\Omega}_{\R \oplus \R} \circ
%\psi = \widetilde{\Omega}_{\C}$ and $\widetilde{\Omega}_{\C} \circ
%\phi = \widetilde{\Omega}_{\R \oplus \R}$.\\
%Moreover, the quotient topology on $ \widetilde{\Omega}_{\R \oplus
%\R}(\Sc(H(\R \oplus \R)))$ coincides with the quotient topology on
%$\widetilde{\Omega}_{\C}(\Sc(H( \C))).$\\
 (ii) $ \widetilde{\Omega}_{\R \oplus \R}(\Sc(S(\R \oplus \R))) =
\widetilde{\Omega}_{\C}(\Sc(S( \C))).$ %Moreover, the quotient
%topology on $ \widetilde{\Omega}_{\R \oplus \R}(\Sc(S(\R \oplus
%\R)))$ coincides with the quotient topology on
%$\widetilde{\Omega}_{\C}(\Sc(H( \C))).$
%Moreover, for any open set
%$U \subset \Sc(S(\R \oplus \R))$, the set
%%
%$\widetilde{\Omega}_{\C}^{-1}(\widetilde{\Omega}_{\R \oplus
%\R}(U))$ is open in $\Sc(S( \C))$ and for any open set $W \subset
%\Sc(S(\C))$, the set $\widetilde{\Omega}_{\R \oplus
%\R}^{-1}(\widetilde{\Omega}_{\C}(W))$ is open in $\Sc(S( \R \oplus
%\R))$.
%Moreover, there exist continuous linear maps $\phi': \Sc(S(\R
%\oplus \R)) \to \Sc(S( \C))$ and $\psi':\Sc(S(\C)) \to \Sc(S( \R
%\oplus \R))$ such that $\widetilde{\Omega}_{\R \oplus \R} \circ
%\psi = \widetilde{\Omega}_{\C}$ and $\widetilde{\Omega}_{\C} \circ
%\phi = \widetilde{\Omega}_{\R \oplus \R}$.
\end{thm}
%\begin{rem}
%note that by banch open map therem it is enaghf to chack the euo
%\end{rem}

\subsection{Main ingredients} \label{SubSecMainIng} $ $

In this subsection we list three main ingredients of the proof of the main theorem.

\subsubsection{Intermediate Kloosterman Integrals}\label{IntKloosInt}

\begin{prop} \label{FirstIng}
$ $\\
(i) The map $\widetilde{\Omega}_i^n$ defines an onto map
$\Sc(O_i^n) \to
\Sc(S^i\times H^{n-i})$.\\ %Moreover, this map has a continuous linear section. \\
(ii) $\widetilde{\Omega}^n = \widetilde{\Omega}^{i,n-i} \circ \widetilde{\Omega}_i^n$.
\end{prop}
\begin{proof}
(i) follows from Property \ref{NashSub}, since the map $\beta_i$ from the proof of Proposition \ref{OmegaWell} is a surjective submersion.\\
(ii) is straightforward.
\end{proof}

\begin{prop}
$ \widetilde{\Omega}^{m,n}(\Sc(S^m\times H^{n}))=\Im
\cM_{\widetilde{\Omega}^{m}(\Sc(S^{m})),\widetilde{\Omega}^{n}(\Sc(H^{n}))}$.
\end{prop}
\begin{proof}
Follows from the fact that
$\widetilde{\Omega}^{m,n}|_{\Sc(S^m\times H^{n})}=
\widetilde{\Omega}^{m}|_{\Sc(S^m)} \hot
\widetilde{\Omega}^{n}|_{\Sc(H^{n})} \circ \cM_{S^m,H^{n}}$ and
Corollary \ref{ImTensorIm}.
\end{proof}
From the last two propositions we obtain the following corollary.
\begin{cor}  \label{FirstIngCor}
$\widetilde{\Omega}^n (\Sc(O_i^n)) = \Im
\cM_{\widetilde{\Omega}^{n-i}(\Sc(S^{n-i})),\widetilde{\Omega}^{i}(\Sc(H^{i}))}$.
\end{cor}

\subsubsection{Inversion Formula}\label{IngInvFor}

\begin{thm}[Jacquet]\label{SecondIng}
\begin{multline*}
\widetilde{\Omega}^{\overline{\psi}}(\Fou(f))(diag(a_1,...,a_n))=\\=c^{n(n-1)/2}\int...\int
\widetilde{\Omega}^{\psi} (f)(diag(p_1,..p_n)) \psi(-\sum_{i=1}^n
a_{n+1-i} p_i +\sum_{i=1}^{n-1} 1/(a_{n-i} p_i) ) dp_n ... dp_1.
\end{multline*}
Here, $c$ is a constant, we will discuss it in \S\S \ref{Sec_wil}. The integral here is just an iterated integral. In particular we
mean that the integral converges as an iterated integral.
\end{thm}
The proof is essentially the same as in the p-adic case (see \cite[Section 7]{Jac1}).
For the sake of completeness we repeat it in \S \ref{sec_inv_for}.
\subsubsection{Key Lemma}\label{IngKeyLem}

\begin{lemma}[Key Lemma]\label{KeyLem}
Consider the action of $N$ on $\Sc(H)$ to be the standard action twisted by $\chi$.
Then
% given by$(u(f))(x)=\chi(u)f(u^{-1}x)$.
% ?? now our acion is right action, but we think about left action!
$$\Sc(H)=\Sc(U)+\Fou(\Sc(U))+\n \Sc(H).$$
\end{lemma}
For proof see \S \ref{SecPfKeyLem}.

\subsection{Proof of the main result} \label{SubSecPfMain}$ $

We prove Theorem \ref{ReformMain} by induction. The base $n=1$ is obvious. Thus, from now on
we assume that $n \geq 2$ and that Theorem \ref{ReformMain}
holds for all dimensions smaller than $n$.

\begin{prop}
$$\widetilde{\Omega}_{\R \oplus \R}(\Sc(O_i(\R
\oplus \R))) = \widetilde{\Omega}_{\C}(\Sc(O_i(\C))).$$
\end{prop}
\begin{proof}
Follows  from Corollary \ref{FirstIngCor} and the induction
hypothesis.
\end{proof}

\begin{cor}
$$\widetilde{\Omega}_{\R \oplus \R}(\Sc(U(\R \oplus \R))) =
\widetilde{\Omega}_{\C}(\Sc(U(\C))).$$
\end{cor}
\begin{proof}
Follows from the the previous proposition and partition of unity
(property \ref{pCosheaf}).
\end{proof}

\begin{cor}
Part (i) of Theorem \ref{ReformMain} holds. Namely, $
\widetilde{\Omega}_{\R \oplus \R}(\Sc(H(\R \oplus \R))) =
\widetilde{\Omega}_{\C}(\Sc(H( \C))).$
\end{cor}
\begin{proof}
By the previous Corollary and Theorem \ref{SecondIng},
$$\widetilde{\Omega}_{\R \oplus \R}(\Fou(\Sc(U(\R \oplus \R)))) =
\widetilde{\Omega}_{\C}(\Fou(\Sc(U(\C)))).$$
Clearly, $\widetilde{\Omega}_{\R \oplus \R}(\n \Sc(H(\R \oplus
\R)))=\widetilde{\Omega}_{\C}(\n \Sc(H(\C)))=0$. Hence, by Remark \ref{SubspaceRem}
%\begin{multline*}
$$\widetilde{\Omega}_{\R \oplus \R}(\Sc(U(\R \oplus \R))+\Fou(\Sc(U(\R \oplus \R)))+\n \Sc(H(\R \oplus \R)))\\
=\widetilde{\Omega}_{\C}(\Sc(U(\C))+\Fou(\Sc(U(\C)))+\n \Sc(H(\C))),$$
%\end{multline*}
where we again consider the action of $N$ on $\Sc(H)$ to be twisted by $\chi$.
Therefore, by the Key Lemma
$$ \widetilde{\Omega}_{\R \oplus \R}(\Sc(H(\R \oplus
\R))) = \widetilde{\Omega}_{\C}(\Sc(H( \C))).$$
\end{proof}
It remains to prove part (ii) of Theorem \ref{ReformMain}.
\begin{proof}[Proof of part (ii) of Theorem \ref{ReformMain}]
%$ \widetilde{\Omega}_{\R \oplus \R}(\Sc(S(\R \oplus \R))) =
%\widetilde{\Omega}_{\C}(\Sc(S( \C))).$
By Property \ref{PropDixMal}, %put it in prelim.
$$\Sc(S(\R \oplus \R))= \Sc(\R^{\times})\Sc(S(\R \oplus \R)),$$
and hence
$$\Sc(S(\R \oplus \R))= \Sc(\R^{\times})\Sc(H(\R \oplus \R)),$$
where the action of $\Sc(\R^{\times})$ on $\Sc(H(\R \oplus \R))$
is given via $\det : H(\R \oplus \R) \to \R$.

Hence $$\widetilde{\Omega}_{\R \oplus \R}(\Sc(S(\R \oplus \R))) =
\Sc(\R^{\times}) \widetilde{\Omega}_{\R \oplus \R}(\Sc(H(\R \oplus
\R))).$$ By part (i) $$\Sc(\R^{\times}) \widetilde{\Omega}_{\R
\oplus \R}(\Sc(H(\R \oplus \R))) = \Sc(\R^{\times})
\widetilde{\Omega}_{\C}(\Sc(H(\C))).$$ As before,
$$\Sc(\R^{\times})
\widetilde{\Omega}_{\C}(\Sc(H(\C))) =
\widetilde{\Omega}_{\C}(\Sc(\R^{\times})\Sc(H(\C)))=
\widetilde{\Omega}_{\C}(\Sc(S(\C))).$$
\end{proof}
%Rami
\begin{rem}
One can give an alternative proof, that does not use
Property \ref{PropDixMal}, in the following way.
Define maps $\widetilde{\Omega}':
\Sc(H \times \R^{\times}) \to C^\infty(A \times \R^{\times})$
similarly to $\widetilde{\Omega}$, and not involving the second coordinate.
From (i), using \S
\ref{NucFre}, we get that $\Im \widetilde{\Omega}'_\C= \Im
\widetilde{\Omega}'_{\R \oplus \R}$. Using the graph of $\det$ we
can identify $S$ with a closed subset of $H \times \R^{\times}$
and $A$ with a closed subset of $A \times \R^{\times}$. By
Property \ref{Extension}, the restriction map $\Sc(H \times
\R^{\times}) \to \Sc(S)$ is onto and hence
$\widetilde{\Omega}(\Sc(S))=\Im res \circ \widetilde{\Omega}'$,
where $res:C^\infty(A \times \R^{\times}) \to C^\infty(A)$ is the
restriction. This implies (ii).

In fact, this alternative proof of (ii) is obtained from the previous proof by replacing Property \ref{PropDixMal}
with its weaker version that states
(in the notations of property \ref{PropDixMal}) that the map
$\Sc(M) \ctp \Sc(N) \to \Sc(M)$ is onto. This is much simpler
version since it follows directly from Property \ref{Extension} and Proposition \ref{SchKer}.
%and the fact that one can extend Schwartz function from a closed subset (by property \ref{Extension}).
\end{rem}
\section{Proof of the inversion formula}\label{sec_inv_for}
In this section we adapt the proof of Theorem \ref{SecondIng} given in \cite{Jac1} to  the Archimedean  case.   The proof is by induction. The induction step is based on analogous formula for the intermediate Kloostermann integral which is based on the Weil formula.

In \S\S \ref{Sec_fou} we give notations for various Fourier transforms on $H$. In \S\S \ref{Sec_wil} we recall the Weil formula and consider its special case which is relevant for us. In \S\S \ref{Sec_Jac} we introduce the Jacquet transform and the intermediate Jacquet transform which appears on the right hand side of the inversion formulas. In \S\S\ref{Sec_par} we prove the intermediate inversion formula. In \S\S \ref{Sec_inv} we prove the  inversion formula.
\subsection{Fourier transform}\label{Sec_fou}
\begin{itemize}
\item We denote by $\Fou':=\Fou'_{H_n}: \Sc(H_n)\to \Sc(H_n)$ the Fourier transform w.r.t. the trace form.
\item Note that $\Fou_{H_n}=ad(w) \circ \Fou'_{H_n}= \Fou'_{H_n} \circ ad(w)$.
\item We denote by $\Fou'_{H_i \times H_{n-i}}: \Sc(H_n)\to \Sc(H_n)$ the partial Fourier transform w.r.t. the trace form on $H_i \times H_{n-i}$.
\item We denote by $(H_i \times H_{n-i})^{\bot'} \subset H_n $ the orthogonal compliment to $H_i \times H_{n-i}$ w.r.t. the trace form.
\item We denote by $\Fou'_{{H_i \times H_{n-i}}^{\bot'}} : \Sc(H_n)\to \Sc(H_n)$ the partial Fourier transform w.r.t. the trace form on ${H_i \times H_{n-i}}^{\bot'}$.
\item Note that $\Fou'_{H_n}= \Fou'_{{H_i \times H_{n-i}}^{\bot'}} \circ \Fou'_{{H_i \times H_{n-i}}}= \Fou'_{{H_i \times H_{n-i}}} \circ \Fou'_{{H_i \times H_{n-i}}^{\bot'}}$.

\end{itemize}
\subsection{The Weil formula} \label{Sec_wil}$ $\\
Let $\psi$ be a non-trivial additive character of $\R$.
Recall the one dimensional Weil formula:

\begin{prop}\label{prop:WeilFor}
Let $ a \in \R^\times$. Consider the function $\xi:D \to \R$ defined by $\xi(x)=\psi(a x \bar{x})$ as a distribution on $D$.
Then $\Fou^*(\xi)=\zeta$, where $\zeta$ is a distribution defined by the function $\zeta(x)=|a|^{-1} \eta_{D}(a)c(D,\psi)\psi(-x \bar{x}/a)$.
\end{prop}
One can take this as a definition of $c(D,\psi)$.

The following proposition follows by a straightforward computation.
\begin{prop}\label{prop:WeilCor} $ $\\
(i) $c(\R \oplus \R,\psi)=1$\\
(ii) $c(\C,\psi)^2=-1$\\
(iii) $c(\C,\psi)c(\C,\overline{\psi})=1$
\end{prop}

Proposition \ref{prop:WeilFor} gives us the following corollary.

\begin{cor}
Let $V$ be a free module over $D$ equipped with a volume form. We have a natural Fourier transform $\Fou^*:\Sc^*(V) \to \Sc^*(V^*).$ Let $Q$  be a hermitian form on $V$. Consider the function $\xi:V \to \R$ defined by $\xi(v)=\psi(Q(v))$ as a distribution on $V$. Let $Q^{-1}$ be a hermitian norm on $V^*$ which is the inverse of $Q$. Let $\det(Q)$ be the determinant of $Q$ with respect to the volume form on $V$. Let $\zeta$ be a distribution defined by the function $$\zeta(x)=|\det(Q)|^{-1}(\eta_{D}(\det(Q)c(D,\psi))^{\dim V}\psi(-Q^{-1}(x)).$$ Then $\Fou^*(\xi)=\zeta.$ \end{cor}
\begin{cor} \label{wil_cor}
Let $(A,B) \in S^i\times S^{n-i}$. Consider the function $\xi:{H_i \times H_{n-i}}^{\bot'} \to \R$ defined by $\xi \left[\begin{pmatrix}0 & \bar u^{t} \\
u & 0 \\
\end{pmatrix} \right]=\psi(BuA\bar u^{t})$ as a distribution on $V$. Consider also the function $\zeta:{H_i \times H_{n-i}}^{\bot'} \to \R$ defined by $$\zeta \left[\begin{pmatrix}0 & \bar u^{t} \\
u & 0 \\
\end{pmatrix} \right]=(\eta(\det A)/|\det A|)^{n-i} (\eta(\det B)/|\det B|)^{i} c(D,\psi)^{(n-i)i}\psi(B^{-1}\bar u^{t} A^{-1} u)$$ as a distribution on $V$.

Then $( \Fou'_{{H_i \times H_{n-i}}^{\bot'}})^*(\xi)=\zeta.$
\end{cor}

\subsection{Jacquet transform}\label{Sec_Jac}
\begin{defn}
Let $\psi$ be a non-trivial additive character of $\R$. Let $0 \leq i \leq n.$
\begin{itemize}
%\item
% we denote $\Fou_{H^{n-i},\psi}:\Sc^{\Delta_i,\R^{\times}}(S^i \times
%H^{n-i}) \to \Sc^{\Delta_i,\R^{\times}}(S^i \times
%H^{n-i})$ the partal fourer transform w.r.t. the fixed quadratic form on %$H^{n-i}$ and the chacter $\psi$
\item We define $\cJ_i':C^\infty(S^i \times
S^{n-i}) \to C^\infty(S^{i} \times
S^{n-i})$ by $\cJ_i'(f)(A,B)=f(A,B)\psi(wB^{-1}w \eps A^{-1} \eps^t)$. Here $\eps$
is the matrix with $n-i$ rows and $i$ columns whose first row is the row
$(0, 0, . . . , 0, 1)$
and all other rows are zero.%,  and $\tilde{\eps}=\eps^t.$
\item We define
%$\cT_{i}':C^\infty(S^i \times S^{n-i}) \to C^\infty(S^{n-i} \times S^{i})$
$\cT_{i}:C^\infty(S^i \times
S^{n-i}) \to C^\infty(S^{n-i} \times
S^{i})$ by $\cT_i(f)(A,B)=f(B,A)$.

\item We denote by $\fJ_{i,n-i}:=\Sc^{\Delta_i,\R^{\times}}(S^i \times
H^{n-i}) \cap \Fou_{H^{n-i},\psi}^{-1} (\cT_{i}
^{-1}(\cJ_{i}'^{-1} (\Sc^{\Delta_{n-i},\R^{\times}}(S^{n-i} \times
H^{i}))
))
$
\item We define the partial Jacquet transform $\cJ_{i}: \fJ_{i,n-i} \to \Sc^{\Delta_{n-i},\R^{\times}}(S^{n-i} \times
H^{i}))$ by $$\cJ _{i} :=\Fou_{H^{i},\psi} \circ \cT_{i} \circ \cJ'_{i} \circ \Fou_{H^{n-i},\psi}|_{\fJ_{i,n-i}}.$$
\item Denote by $\overline{A}$ the set of diagonal matrices in $H$.
\item We denote $\Fou_n : \Sc^{\Delta_{n-1}}(\overline{A}) \to \Sc^{\Delta_{n-1}}(\overline{A})$ the Fourier transform w.r.t. the last co-ordinate.
\item We define $${\cJ_n^{(i)}}':\Sc^{\Delta_{n-1}}(\overline{A}) \to \C^{\infty}(A)$$
by $${\cJ_n^{(i)}}'(f)(a_1,...,a_n)=f(a_1,...,a_{i-1},a_n,a_{i}..,a_{n-1})\psi(1/a_n a_{n-1}).$$

\item We define $\cJ_n^{(i)}:\Sc^{\Delta_{n-1}}(\overline{A}) \to \C^{\infty}(A)$ by $\cJ_n^{(i)}={\cJ_n^{(i)}}' \circ \Fou_n$ for $i<n$ % and by $\cJ_n^{(n)}=\Fou_n$

\item We define inductively a sequence of subspaces $\fJ_n^{[i]} \subset \C^{\infty}(A)$
and operators $\cJ_n^{[i]}: \fJ_n^{[i]} \to \C^{\infty}(A)$ in the following way $\fJ_n^{[1]}= \Sc^{\Delta_{n-1}}(\overline{A})$, $\cJ_n^{[1]}= \Fou_n$, $\fJ_n^{[i]}= \Sc^{\Delta_{n-1}}(\overline{A}) \cap (\cJ_n^{(i)})^{-1}(\fJ_n^{[i-1]})$ and $\cJ_n^{[i]}= \cJ_n^{[i-1]} \circ \cJ_n^{(n+1-i)}.$
\item We define the Jacquet space $\fJ := \fJ_n$ to be $\fJ_n^{[n]}$ and the Jacquet transform $\cJ := \cJ_n:\fJ \to C^{\infty}(A)$ to be $\cJ_n^{[n]}.$

\end{itemize}
\end{defn}

\subsection{The partial inversion formula}\label{Sec_par} $ $\\
In this subsection we prove an analog of Proposition 8 of \cite{Jac1}, namely
\begin{prop}\label{prop_par_inv}$ $
\\
(i) $\wO_i(\Sc(H)) \subset \fJ_{i,n-i}$
\\
(ii) $\cJ_{i} \circ \wO_i^\psi|_{\Sc(H)} =  c(D,\psi)^{n(n-i)}\wO_{n-i}^{\bar \psi} \circ \Fou_{H} $
\end{prop}
This proposition is equivalent to the following one
\begin{prop} \label{prop_par_inv2}
$$ \cJ'_{i} \circ \Fou_{H^{n-i},\psi}\circ\wO_i^\psi|_{\Sc(H)}=  c(D,\psi)^{n(n-i)}\cT_{i}
^{-1}\circ(\Fou_{H^{i},\psi} )^{-1}\circ\wO_{n-i}^{\bar \psi} \circ \Fou.$$
\end{prop}

For its proof we will need some auxiliary results.
\begin{lemma}
Let $f \in {\Sc(H)}$ be a Schwartz  function. Then
$$\wO_i^\psi(f)(A,B)=\eta(\det(A))^{n-i} |\det(A)|^{-(n-i)} \int_{} f \left[ \begin{pmatrix}A & X \\
\bar X^t &B+ X^t A^{-1}X\\
\end{pmatrix} \right]\psi[\tr(\eps A^{-1} X)+\tr( \bar X^{t} A^{-1 }\eps^t)]dX$$
\end{lemma}
The proof is straightforward.
%This lemma gives as the folowing two corollaries
\begin{cor} \label{cor:par_trans}
Let $f \in {\Sc(H)}$ be a Schwartz  function. Then
\begin{multline*}\Fou_{H^{n-i},\psi}\circ\wO_i^\psi(f)(A,w_{n-i}Cw_{n-i})=
\eta(\det(A))^{n-i} |\det(A)|^{-(n-i)}
\\
\int f \left[ \begin{pmatrix}A & X \\
\bar X^t &B\\
\end{pmatrix}\right] \psi[\tr(\eps A^{-1} X)+\tr( \bar X^{t}  A^{-1 } \eps^t)+\tr(C X^t A^{-1}X)-Tr(CB)]dX dB\end{multline*}
\end{cor}
\begin{notation}
$ $\\
(i)
Let $\xi_{A,B} \in \Sc^*(H)$ be the distribution defined by $$\xi_{A,B}(f)= \cJ'_{i} \circ \Fou_{H^{n-i},\psi}\circ\wO_i^\psi(f)(A,B).$$
(ii)
Let $\zeta_{A,B} \in \Sc^*(H)$ be the distribution defined by $$\zeta_{A,B}(f)= \cT_{i}
^{-1}\circ(\Fou_{H^{i},\psi} )^{-1}\circ\wO_{n-i}^{\bar \psi}(f)(A,B).$$
\end{notation}

\begin{proof}[Proof of Proposition \ref{prop_par_inv2}]
We have to show that $$\xi_{A,B}=  c(D,\psi)^{n(n-i)}\Fou(\zeta_{A,B})$$
Let $f \in {\Sc(H)}$ be a Schwartz  function. Denote $m:=n-i.$ By Corollary \ref{cor:par_trans}
\begin{multline*}\xi_{A,C}(f)=\eta(\det(A))^{n-i} |\det(A)|^{-(n-i)}\psi(w_{n-i}C^{-1}w_{n-i} \eps A^{-1} {\eps^t})
\\
\int f \left[ \begin{pmatrix}A & X \\
\bar X^t &B\\
\end{pmatrix}\right] \psi[\tr(\eps A^{-1} X)+\tr( \bar X ^{t} A^{-1} \eps^t)+\tr(w_{n-i}Cw_{n-i} X^t A^{-1}X)-Tr(w_{n-i}Cw_{n-i}B)]dX dB
\end{multline*}
and
\begin{multline*}
\zeta_{A,C}(f)=
\eta(\det(C))^{i} |\det(C)|^{-i} \times
\\
\int_{} f \left[ \begin{pmatrix}C & X \\
\bar X^t &B\\
\end{pmatrix}\right] \psi[-\tr(\eps C^{-1} X+ \bar X^{t}  C^{-1 } \eps^t+w_{i}Aw_{i} X^t C^{-1}X-w_{i}Aw_{i}B)]dX dB.
\end{multline*}

Therefore
\begin{multline*}
ad(w_n)(\zeta_{A,C})(f)=
\eta(\det(C))^{i} |\det(C)|^{-i} \times
\\
\int_{} f \left[ \begin{pmatrix}B & X \\
\bar X^t &w_{m}Cw_{m}\\
\end{pmatrix}\right] \psi[-\tr(\eps C^{-1} w_{m}\bar X^{t}w_{m}+w_{m} X w_{m} C^{-1 } \eps^t+A_{} X w_{m}C^{-1}w_{m}\bar X^t- AB)]dX dB.
%\Dima I put \tr out of paranthesys so that the integral will be inside one line.
\end{multline*}
Thus
% \begin{multline*}
% \Fou'_{H_i \times H_m} (ad(w_n)(\zeta_{A,C}))(f)=
% \eta(det(C))^{i} |det(C)|^{-i} \times
% \\
% \int_{} f \left[ \begin{pmatrix}A & X \\
% \bar X^t &Bw_{m}Cw_{m}\\
% \end{pmatrix}\right] \psi[-\tr(\eps C^{-1} w_{m}\bar X^{t}w_{m}+w_{m} X w_{m} C^{-1 } \eps^t+A_{} X w_{m}C^{-1}w_{m}\bar X^t+w_{m}Cw_{m}B)]dX dB.
% %\Dima I put \tr out of paranthesys so that the integral will be inside one line.
% \end{multline*}

\begin{multline*}
\Fou'_{H_i \times H_m} (ad(w_n)(\zeta_{A,C}))(f)=
\eta(\det(C))^{i} |\det(C)|^{-i} \times
\\
\int_{} f \left[ \begin{pmatrix}A & X \\
\bar X^t &B\\
\end{pmatrix}\right] \psi[-\tr(\eps C^{-1} w_{m}\bar X^{t}w_{m}+w_{m} X w_{m} C^{-1 } \eps^t+A_{} X w_{m}C^{-1}w_{m}\bar X^t+w_{m}Cw_{m}B)]dX dB.
%\Dima I put \tr out of paranthesys so that the integral will be inside one line.
\end{multline*}

Therefore by Corollary \ref{wil_cor}  $$\Fou'_{{H_i \times H_m}^{\bot'}} (\Fou'_{H_i \times H_m} (ad(w_n)(\zeta_{A,C})))(f)=  c(D,\psi)^{n(n-i)}\xi_{A,C}(f).$$
\end{proof}
\subsection{Proof of the inversion formula} \label{Sec_inv}$ $\\
The inversion formula (Theorem \ref{SecondIng}) is equivalent to the following theorem.
\begin{thm}  \label{SecondIng2}
$ $
\\
(i) $\wO(\Sc(H)) \subset \fJ$.
\\
(ii) $\cJ_{} \circ \wO^\psi|_{\Sc(H)} =  c(D,\psi)^{n(n-1)/2}\wO^{\bar \psi} \circ \Fou_{H}$.
\end{thm}

The proof is by induction.
We will need the following straightforward lemma.

\begin{lem}\label{ind_step}
 The induction hypotheses implies that\\
 (i) $\wO^{1,n-1}({\Sc^{\Delta_1}(S^1\times H^{n-1})}) \subset \fJ_n^{[n-1]}$ \\
 (ii) $$
\wO^{1,n-1,\bar \psi}\circ\Fou_{H^{n-1},\psi} = c(D,\psi)^{(n-1)(n-2)/2}\cJ_n^{[n-1]} \wO^{1,n-1,\psi}|_{\Sc^{\Delta_1}(S^1\times H^{n-1})}$$
\end{lem}

\begin{proof}[Proof of Theorem \ref{SecondIng2}]
First let us prove  (i).
It is easy to see that
\begin{equation}\label{eq0}
\wO^{1,n-1,\psi}|_{\Sc^{\Delta_1}(S^1\times H^{n-1})} \circ \cT_{n-1} \circ \cJ'_{n-1}|_{ \Fou_{H^{1},\psi}(\fJ_{n-1,1})}={\cJ_n^{(i)}}' \circ\wO^{1,n-1,\psi} |_{ \Fou_{H^{1},\psi}(\fJ_{n-1,1})}
\end{equation}
This implies that
\begin{equation}\label{eq1}
\wO^{1,n-1}|_{\Sc^{\Delta_1}(S^1\times H^{n-1})} \circ \cT_{n-1} \circ \cJ'_{n-1}\circ \Fou_{H^{1},\psi}\circ \wO_{n-1}|_{\Sc(H)}={\cJ_n^{(i)}}' \circ \Fou_{n} \circ\wO^{1,n-1,\psi} \circ \wO_{n-1}|_{\Sc(H)}
\end{equation}

By Proposition \ref{FirstIng} this implies
\begin{equation}\label{eq2}
\wO^{1,n-1}|_{\Sc^{\Delta_1}(S^1\times H^{n-1})} \circ \cT_{n-1} \circ \cJ'_{n-1}\circ \Fou_{H^{1},\psi}\circ \wO_{n-1}|_{\Sc(H)}={\cJ_n^{(i)}}' \circ \Fou_{n} \circ \wO|_{\Sc(H)}
\end{equation}
This together with Lemma \ref{ind_step}  implies (i).

Now let us prove (ii). By Propositions \ref{FirstIng}
and \ref{prop_par_inv} we have
%\begin{equation} \label{eq3}
\begin{multline} \label{eq3}
\wO^{\bar \psi} \circ \Fou_{H}=\wO^{1,n-1,\bar \psi}\circ \wO^{\bar \psi}_{1}\circ \Fou_{H}=  c(D,\psi)^{(n-1)}\wO^{1,n-1,\bar \psi}\circ\cJ_{n-1} \circ \wO_{n-1}^\psi|_{\Sc(H)}=\\= c(D,\psi)^{(n-1)}\wO^{1,n-1,\bar \psi}\circ\Fou_{H^{n-1},\psi} \circ \cT_{n-1} \circ \cJ'_{n-1} \circ \Fou_{H^{1},\psi}\circ \wO_{n-1}^\psi|_{\Sc(H)}
\end{multline}
%\end{equation}

(ii) follows now from (\ref{eq2}), (\ref{eq3}), and Lemma \ref{ind_step}.
\end{proof}
\section{Proof of the Key Lemma} \label{SecPfKeyLem}
%Fix $n \geq 2$. In this section we fix a semi-simple algebra $D$
%over $F$ of dimension 2. We denote by $H:=H_n(D/F)$ the space
%hermitian matrices of size $n$. Denote $\Delta_i:H \to F$ be the
%main $i$-minor. Let $O_i \subset H$ be the subset of matrices with
%$\Delta_i \neq 0$. Let $Z := Z_n(D/F):=H - \cup O_i$. Let
%$N_n:=N_n(D)<GL_n(D)$ be the subgroup consisting  of upper
%unipotent matrices. We define a character $\theta : N_n(F) \to
%\C^{\times}$ by $\theta(x) := \psi(\sum (x_{i,i+1} + \overline{
%x_{i,i+1}}))$. Let the group $N_n \times N_n$  act on $H$ by $g
%\mapsto \on^t x n$. Fix a symmetric $F$-bilinear form on $H$ by
%$B(x,y):= \tr_F(x w yw)$, where $w$ is the longest element in the
%Weyl group.
%
%In this section we prove the following lemma.
%
%\begin{lemma}[Key Lemma]\label{KeyLem2}
%Suppose that $\xi \in \Sc^*(H)^{N,\psi}$ such that $\Supp\xi,
%\Supp\widehat{\xi} \subset Z.$  Then $\xi =0$.
%\end{lemma}
\setcounter{lemma}{0}

We will use the following notation and lemma.

\begin{notation}
Denote
$$Z':= \{x
\in Z | x_{ij} = 0 \text{ for  }i+j < n+1  \text{ and }
x_{i,n+1-i} = x_{j,n+1-j} \in \R \text{ for any } 1 \leq i,j \leq n
\} .$$
Denote also $U':=H -Z'$.
\end{notation}

%\begin{lemma}\label{RelSuppZ'}
%Every $\xi \in \Sc^*_H(Z)^{N,\psi}$ is supported in $Z'$.
%\end{lemma}
%
%To prove this lemma we will use the following definition and
%lemma.

\begin{notation}
We call a matrix $x \in H$ \textit{relevant} if $\chi|_{N_x} \equiv
1$, and irrelevant otherwise.
\end{notation}
%\Rami{we will use the folowing lema}

\begin{lemma}[\cite{Jac1}, \S3, \S5] \label{lem:RelForm}
Every relevant orbit in $H^n(D)$ has a unique representative of the form
\begin{equation} \label{eq:RelForm0}
\begin{pmatrix}
  a_1 w_{m_1} & 0 & ... & 0 \\
  0 & a_2 w_{m_2} & ... & 0 \\
  ...& ... & ... & ... \\
  0 & 0 & ... & a_n w_{m_n}
\end{pmatrix}
\end{equation}
where $m_1+...+m_j=n$, $a_1,...,a_j \in \bR$, and if $\det(g)=0$ then $\Delta_{n-1}(g) \neq 0$.
\end{lemma}
For the sake of completeness we will repeat the proof here.
%For the proof of this lemma see \cite[\S2, pp124-125 and \S 5, pp. 138-139]{Jac1}.
\begin{proof}
%Let $g \in  H^n(D)$ be relevant.
Step 1. Proof for $S^n(\R \oplus \R)$\\%the case $D= \R \oplus \R$, $\det(g) \neq 0$
Let $W_n$ denote the group of permutation matrices.
By Bruhat decomposition, every orbit has a unique representative of the form $wa$ with $w \in W_n$ and $a \in A^n$.
If this element is relevant, then for every pair of positive roots $(\alp_1, \alp_2)$ such that
$w\alp_2 = - \alp_1$, and for $u_{i} \in  N_{\alp_i}(\R)$ (where $N_{\alp_i}$ denotes the one-dimensional subgroup of N corresponding to ${\alp_i}$) we have
\begin{equation} \label{eq:rel}
u_1^t wa u_2 = wa  \Rightarrow \chi(u_1,u_2) = 0.
\end{equation}
This condition implies that $\alp_1$ is simple if and only if $\alp_2$ is simple. Thus $w$
and its inverse have the property that if they change a simple root to a negative one,
then they change it to the opposite of a simple root. Let $S$ be the set of simple roots $\alp$
such that $w \alp$ is negative. Then $S$ is also the set of simple roots $\alp$ such that $w^{-1}\alp$ is
negative and $wS = −S$. Let $M$ be the standard Levi subgroup determined by $S$. Thus
$S$ is the set of simple roots of $M$ for the torus $A$, $w$ is the longest element of the Weyl group of $M$,
and $w^2 = 1$. This being so, if $\alp_2$ is simple, then condition (\ref{eq:rel}) implies $\alp_2(a) = 1$.
Thus $a$ is in the center of $M$. Hence $wa$ is of the form (\ref{eq:RelForm0}).

Step 2. Proof for $S^n(\C)$.\\
Every orbit has a unique representative of the form $wa$ with
$w \in W_n$, and diagonal $a \in  GL_n(\C)$ (for proof see e.g. \cite[Lemma 4.1(i)]{Spr}, for the involution $g \mapsto w_n\overline{g}^{-t}w_n$, where $w_n \in W_n$ denotes the longest element). Since $wa \in S$, we have $w = w^t$ and hence $w^2 = 1$  and $waw = \overline{a}$.

Suppose that $\alp$ is a simple
root such that $w\alp = - \bet$ where $\bet$ is positive. For $u_{\alp}\in N_{\alp}$, define
$$u_{\bet}  := w \overline{a}^{-1} \ou_{\alp}^{-t} \overline{a} w  \in N_{\bet}.$$
%by $\ou_{\bet}^twau_{\alp} = wa.$
Then
$$\ou_{\bet}^twau_{\alp} = wa = \ou_{\alp}^twau_{\bet}.$$
There exists an element $u_{\alp + \bet} \in N_{\alp + \bet}$ (i.e. $u_{\alp + \bet} = 1 $ if ${\alp + \bet}$ is not a root) such that
$u: = u_{\alp + \bet}u_{\alp}u_{\bet}$ satisfies
$\ou^t wa u = wa$.
If $wa$ is relevant, this relation implies
$\chi(u_{\alp}u_{\bet}) = 1$.

Thus $\beta$ is simple. Since $w^2 = 1$, we
see that, as before, there is a standard Levi subgroup $M$ such that w is the longest
element in its Weyl group, and $a \in Z(M)\cap A^n$.

Step 3. Proof for $H^n(D) - S^n(D)$.\\
Let $s \in H^n(D)$ with $\det(s)=0$ be relevant. Then $s = u^twb$ with $u \in N(D)$, $w \in W_n$ and $b$ upper triangular. If a column of $s$ of index $i<n$ would be zero, then the row with index $i$ would also be zero, and hence $s$ would be irrelevant. Hence $b_{1,1}\neq 0$ and acting on $s$ by $N(D)$ we can bring $b$ to the form
$ b = \begin{pmatrix}
  b' & 0 \\
  0 & 0
\end{pmatrix},$
where $b'$ is diagonal and invertible.
In particular, the last row of $b$ is zero. We may replace $s$ by $wb\ou^{-1}$.  The last row of $b\ou^{-1}$ is again zero. Since the rows of $wb\ou^{-1}$ with index less than $n$ cannot be zero, $w$ must have the form
$ w = \begin{pmatrix}
  w' & 0 \\
  0 & 0
\end{pmatrix}.$ The theorem follows now from the 2 previous cases.
\end{proof}

Since $Z$ and $Z'$ are $N$-invariant we obtain
\begin{cor}%[Key Lemma]
Every relevant $x \in Z$ lies in $Z'$.
\end{cor}

%\begin{lemma}%[Key Lemma]
%Let $x \in H$ be such that $x_{ij} = 0 \text{ for }i+j < n+1$.
%Suppose that $x$ is relevant. Then $x_{i,n+1-i} = x_{j,n+1-j} \in
%F$ for any $1 \leq i,j \leq n$.
%\end{lemma}
%
%
%\begin{proof}%[Key Lemma]
%Let $w$ denote the Weyl element (i.e. the longest element in the
%Weyl group). Then $wx \in B$ and $wx$ is relevant if and only if
%$x$ is. Let $z$ denote the diagonal part of $wx$. Then $z$ is
%relevant if and only if $x$ is.  Now, if $z$ is relevant then
%$\theta(z(Id + E_{i,i+1})) = \theta((Id + E_{i,i+1})z)$ for any
%$i$. Therefore, $z_{i,i}=z_{i+1,i+1}$ for any $i$ which means that
%$z$ is a scalar and hence all "anti-diagonal" elements of $x$ are
%equal.
%\end{proof}

Using Corollary
\ref{cor:NegCoinvPrel} we obtain
\begin{cor}\label{RelSuppZ'}
Recall that we consider the action of $N$ on $\Sc(H)$ to be the standard action twisted by $\chi$. Then
$\Sc(U')=\Sc(U) + \n \Sc(U').$
\end{cor}

%\begin{proof}%[Proof of Lemma \ref{RelSuppZ'}]
%It is enough to show that $ \Sc^*_{H}(Z - Z')^{N,\psi}=0$.
%This corollary follows from the previous lemma and
%the fact that the
% Borel subgroup $B$ acts on $H$ with finite number of orbits, and
% $N$ is normal in $B$, using
% %Theorem \ref{NegCoinvPrel}
%\Dima{
%Corollary \ref{cor:NegCoinvPrel}.
%\Dima{We do not need $B$ here any more}.
%
%Since the Borel group $B$ acts on $H$ with finite number of
%orbits, %and $N$ is normal in $B$ By
%Lemma \ref{Orbitwise} implies that it is enough to show that for
%any $x \in Z - Z'$ we have
%$$ (\Sym^k(N_{Nx,x}^H))^{N_x,\psi \Delta N \Delta^{-1}N_x} =0.$$ % Maybe inverse ratio
%Since the groups $N$ and $N_x$ are unipotent, $\Delta N
%\Delta^{-1}N_x = 1$ and the action of $N_x$ on
%$(\Sym^k(N_{Nx,x}^H))$ is unipotent.
%
%By the previous Lemma, $\psi|_{N_x}$ is non-trivial. Hence
%$$ (\Sym^k(N_{Nx,x}^H))^{N_x,\psi \Delta N \Delta^{-1}N_x} =0$$
%and hence $ \Sc^*_{H}(Z - Z')^{N,\psi}=0$.
%\end{proof}

\begin{lemma}
$Z' \nsupseteq Z'^{\bot}$.
\end{lemma}
\begin{proof}
For $n>2$ this is obvious since $\dim Z' < \frac{n^2}{2} =
\frac{\dim H}{2}$.

For $n=2$, $\dim Z' = \frac{n^2}{2} = \frac{\dim H}{2}$. Hence it
is enough to show that $Z' \neq (Z')^{\bot}$. Now
$$B \left (\begin{pmatrix}
  0 & a \\
  a & b
\end{pmatrix}   , \begin{pmatrix}
  0 & c \\
  c & d
\end{pmatrix}  \right ) = 2 ac ,$$
which is not identically 0.
\end{proof}

\begin{cor}\label{Cor2}
$\Sc(H)=\Sc(U')+\Fou(\Sc(U')).$
\end{cor}
\begin{proof}
Follows from the previous lemma and Theorem \ref{DualUncerPrel}.
\end{proof}

\begin{proof}[Proof of the Key Lemma (Lemma \ref{KeyLem})]
By Corollaries \ref{RelSuppZ'} and \ref{Cor2},
\begin{multline*}
\Sc(H) = \Sc(U')+\Fou(\Sc(U'))= \Sc(U) + \n \Sc(U') +
\Fou(\Sc(U) + \n \Sc(U')) =\\= \Sc(U) + \n \Sc(U') +
\Fou(\Sc(U)) + \n \Fou(\Sc(U')) =\\
=  \Sc(U) + \Fou(\Sc(U)) +
\n(\Sc(U')+\Fou(\Sc(U'))) \subset \Sc(U) + \Fou(\Sc(U)) +
\n(\Sc(H)).
\end{multline*}

The opposite inclusion is obvious.
\end{proof}

\section{Non-regular Kloostermann integrals} \label{sec:SingOrbitInt}
\setcounter{lemma}{0}

In this section we define Kloostermann integrals over relevant non-regular orbits. We prove that if two functions match  then their non-regular Kloostermann integrals also equal, up to a matching factor.
We also prove that if all regular Kloostermann integrals of a function vanish then all Kloostermann integrals of this function vanish. In the non-Archimedean case this was done in \cite{JacSing} and the proofs we give  here are very similar.

Recall that $g\in H^n(D)$ is called relevant if the character $\chi$ is trivial on the stabilizer $N(D)_g$ of $g$.
For every relevant $g\in H^n(D)$ and every $\Psi \in \Sc^{\det,\R^{\times}}(S^n(D)) $
we define
$$ \Omega_{D}^{n,\psi}(\Psi,g):= \int_{N/N_g} \Psi(\ou^t a u) \chi(u)du.$$
%Here, $du$ is the standard Haar measure on $N$.
%?? We should think whether we want to define \widetilde{Omega}, or define transfer factors.
Recall the description of relevant orbits given in Lemma \ref{lem:RelForm}:
every relevant orbit in $H^n(D)$ has a unique element of the form
\begin{equation} \label{eq:RelForm}
g = \begin{pmatrix}
  a_1 w_{m_1} & 0 & ... & 0 \\
  0 & a_2 w_{m_2} & ... & 0 \\
  ...& ... & ... & ... \\
  0 & 0 & ... & a_n w_{m_n}
\end{pmatrix},
\end{equation}
where $m_1+...+m_j=n$, $a_1,...,a_j \in \bR$, and if $\det(g)=0$ then $\Delta_{n-1}(g) \neq 0$.
In particular, $H^n(\C)$ and $H^n(\R \oplus \R)$ have the same set of representatives of regular orbits.
%\end{lemma}
%This lemma follows from the Bruhat decomposition by straightforward computations.
%\Dima{?? I really did the computations. DO we want me to write them?

\begin{notation}
We extend the definition of the transfer factor $\gamma$ to all $g$ of the form (\ref{eq:RelForm}) by
\begin{equation} %\label{eq:RelForm}
\text{For }g = \begin{pmatrix}
  x & 0 \\
  0 & y
\end{pmatrix} \in S^i(\C)\times H^{n-i}(\C) \quad \gamma(g) = \gamma(x)\gamma(y) \sign(\det(x))^i
\end{equation}
\begin{equation} %\label{eq:RelForm}
\gamma(a w_n) = \gamma(-a^{-1}w_{n-1}, \overline{\psi})c(\C,\psi)^{n(n-1)/2}\sign(\det(-a^{-1}w_{n-1}))
\end{equation}

\begin{remark}
 Since $c(\C,\psi)^2=-1$ and $c(\C,\psi) c(\C,\overline{\psi}) =1$, we have $\gamma(a w_{n+8}) = \gamma(a w_n)$ and for $1 \leq n \leq 8,$ $\gamma(a w_n)$ is determined by the sequence
$$ 1, \,\, c(\C,\psi) \sign(-a), \,\, \sign(a),\,\, 1, \,\, -1, \,\,  c(\C,\psi) \sign(-a), \,\, \sign(-a), \,\, 1.$$
In particular $\gamma(g)$ is always a fourth root of unity.
\end{remark}
\end{notation}

\begin{theorem}\label{thm:SingMatching}
 Let $\Phi \in \Sc^{\det,\R^{\times}}(H^n(\R \oplus \R )) $ and $\Psi \in \Sc^{\det,\R^{\times}}(H^n(\C))$.
Suppose that $$ \Omega_{\R\oplus \R}^{n,\psi}(\Phi)= \gamma \Omega_{\C}^{n,\psi}(\Psi).$$
Then for any $g$ of the form (\ref{eq:RelForm}) we have $$\Omega_{\R\oplus \R}^{n,\psi}(\Phi,g)= \gamma(g,\psi) \Omega_{\C}^{n,\psi}(\Psi,g).$$
\end{theorem}
For proof see \S\S\ref{subsec:PfSingMatching}.

By substituting $0$ in place of $\Phi$ or $\Psi$ we obtain the folowing corollary
%\begin{theorem}\label{thm:SingDensity}
\begin{cor}[Density]\label{cor:SingDensity}
Let $\Phi \in \Sc^{\det,\R^{\times}}(H^n(D)) $. Suppose that $ \Omega_{D}(\Phi)=0$. Then $\Omega_{D}(\Phi,g)= 0$ for any relevant $g \in D$.
%\end{theorem}\
\end{cor}
%For proof see \S\S\ref{subsec:PfSingDensity}.

For the proof of
%both theorems
Theorem \ref{thm:SingMatching}
we will need the following lemma, which is a more elementary version of the inversion formula.

\begin{lemma}\label{lem:SimpInv}
Let $n>1$. For any $\Phi \in \Sc(H^n(D))$, define the function
$f_{\Phi}$ on $\R^{\times}$ by $f_{\Phi}(a):=\Omega_D(\Phi,aw_n)$. Then $f_{\Phi} \in \Sc(\R^{\times})$ and

% $$f_{\Phi}(a):=|a|^{-n^2+1}\int \Omega_D^{n,\overline{\psi}}(\Fou(\Phi), %\begin{pmatrix}
%        -a^{-1}w_n & 0\\
%         0 & b
%       \end{pmatrix}) db.$$
$$f_{\Phi}(a)=|a|^{-n^2+1}\int \Omega_D^{n,\overline{\psi}}(\Fou(\Phi), \begin{pmatrix}
       -a^{-1}w_{n-1} & 0\\
        0 & b
      \end{pmatrix}) db.$$
\end{lemma}

\subsection{Proof of Lemma \ref{lem:SimpInv}}\label{subsec:PfSimpInv} $ $
\begin{notation}
$ $
\begin{itemize}
\item We denote $V:=\{\{a_{i,j}\}\in H|a_{i,j}=0 \text{ if }i+j\leq n+1\}\subset H$. %?? It was <
\item Note that $V^\bot=\{\{a_{i,j}\}\in H|a_{i,j}=0 \text{ if }i+j < n +1 \}\subset H$. %?? It was <n
\item We denote $e:=\{e_{i,j}\}\in H.$ where $e_{i,j}=\delta_{i+j,n}$.
\end{itemize}
The following two lemmas follow from change of variables.
\begin{lemma} \label{lem:FirstInt}
We have
$$f_{\Phi}(a)=|a|^{(n-n^2)/2}\int_{v \in V} \Phi(aw_n+v) \psi(<a^{-1}e,v>) dv$$
\end{lemma}
\begin{lemma} \label{lem:SecInt}
We have
% $$|a|^{-n^2+1}\int \Omega_D^{n,\bar \psi}(\Fou(\Phi), \begin{pmatrix}
%        -a^{-1}w_{n-1} & 0\\
%         0 & b
%       \end{pmatrix}) db=|a|^{(n-n^2)/2}\int_{v \in V^{\bot}} \Fou(\Phi)(-a^{-1} e+v) \psi(<a^{}w,v>) dv.$$
%
%
% $$|a|^{-n^2+1}\int \Omega_D^{n,\bar \psi}(\Phi, \begin{pmatrix}
%        -a^{-1}w_{n-1} & 0\\
%         0 & b
%       \end{pmatrix}) db=|a|^{(n-n^2)/2}\int_{v \in V^{\bot}} \Phi(-a^{-1} e+v) \psi(<a^{}w,v>) dv.$$
%
% $$\int \Omega_D^{n,\bar \psi}(\Phi, \begin{pmatrix}
%        -a^{-1}w_{n-1} & 0\\
%         0 & b
%       \end{pmatrix}) db=|a|^{(n-n^2)/2+n^{2}-1}\int_{v \in V^{\bot}} \Phi(-a^{-1} e+v) \psi(<a^{}w,v>) dv.$$
%
% $$\int \Omega_D^{n, \psi}(\Phi, \begin{pmatrix}
%        -a^{-1}w_{n-1} & 0\\
%         0 & b
%       \end{pmatrix}) db=|a|^{(n-n^2)/2+n^{2}-1}\int_{v \in V^{\bot}} \Phi(-a^{-1} e+v) \psi(<-a^{}w,v>) dv.$$
%
% $$\int \Omega_D^{n, \psi}(\Phi, \begin{pmatrix}
%        a^{}w_{n-1} & 0\\
%         0 & b
%       \end{pmatrix}) db=|a|^{-(n-n^2)/2-n^{2}+1}\int_{v \in V^{\bot}} \Phi(a^{} e+v) \psi(<a^{}^{-1}w,v>) dv.$$

$$\int \Omega_D^{n, \psi}(\Phi, \begin{pmatrix}
       a^{}w_{n-1} & 0\\
        0 & b
      \end{pmatrix}) db=|a|^{-(n+n^2)/2+1}\int_{v \in V^{\bot}} \Phi(a^{} e+v) \psi(<a^{-1}w,v>) dv.$$

\end{lemma}

\end{notation}
\begin{lemma}
The function $f_{\Phi}$ is in  $\Sc(\R^{\times})$.
\end{lemma}
\begin{proof}
Let $W=\Span(w_n)\oplus V$. Let $\Xi=\Phi|_{W}\in\Sc(W)$. Let $\hat \Xi_V \in\Sc(\Span(w_n)\oplus V^*)$ be the partial Fourier transform of $\Xi$ w.r.t. $V$. For any $a \in \R^{\times} $ let $\phi(a)\in V^*$ be the functional defined by $\phi(a)(v)= <ae,v>$. Consider the closed embedding $\varphi:\R^{\times} \to  \Span(w_n)\oplus V^*)$ defined by $\varphi(a)=(a,\phi(a^{-1}))$. Now by Lemma \ref{lem:FirstInt}, $f_\Phi=\hat \Xi_V \circ \varphi \in \Sc(\R^{\times})$.
\end{proof}
\begin{proof}[Proof of Lemma \ref{lem:SimpInv}]
It is left to prove that
$$f_{\Phi}(a)=|a|^{-n^2+1}\int \Omega_D^{n,\overline{\psi}}(\Fou(\Phi), \begin{pmatrix}
       -a^{-1}w_{n-1} & 0\\
        0 & b
      \end{pmatrix}) db.$$
Let $\delta_{ae+V}\in \Sc(H)$ and $\delta_{aw_n+V^{\bot}} \in \Sc(H)$ be the Haar measures on ${ae+V}$ and $aw_n+V^{\bot}$ correspondingly. Let $f_a,g_a\in C^\infty(H)$ be defined by $f_a(x)=\psi(<ae,x>)$ and $g_a(x)=\psi(<aw_n,x>).$ By Lemmas \ref{lem:FirstInt} and \ref{lem:SecInt} the assertion follows from the fact that $$\delta_{ae+V}g_{-a^{-1}}=\Fou^*(\delta_{-a^{-1}w_n+V^{\bot}}f_a).$$

\end{proof}
\subsection{Proof of Theorem \ref{thm:SingMatching}}\label{subsec:PfSingMatching} $ $

We prove the theorem by induction on $n$. From now on we suppose that it holds for every $r<n$.

\begin{lemma}
 It is enough to prove Theorem \ref{thm:SingMatching} for the case $\Phi \in \Sc(H^n(\R \oplus \R )) $ and $\Psi \in \Sc(H^n(\C))$.
\end{lemma}
\begin{proof}
 Suppose that there exist %. Namely, suppose that there exist matching
$\Phi \in \Sc^{\det,\R^{\times}}(H^n(\R \oplus \R )) $ and $\Psi \in \Sc^{\det,\R^{\times}}(H^n(\C))$ that form a counterexample for  Theorem \ref{thm:SingMatching}. We have to show that then there exist $\Phi ' \in \Sc(H^n(\R \oplus \R )) $ and $\Psi '\in \Sc(H^n(\C))$ that also form a counterexample.

We have $ \Omega_{\R\oplus \R}^{n,\psi}(\Phi)=\gamma\Omega_{\C}^{n,\psi}(\Psi)$ but $\Omega_{\R\oplus \R}^{n,\psi}(\Phi,g)\neq \gamma(g,\psi) \Omega_{\C}^{n,\psi}(\Psi,g)$ for some $g$.
Let $f \in C_c^{\infty}(\R)$ such that $f(\det(g))=1.$
Let $f':= f \circ \det$, and define $\Phi' :=f' \Phi$ and $\Psi':=f' \Psi$.
Note that $\Phi'$ and $\Psi'$ are Schwartz functions and form a counterexample since determinant is invariant under the action of $N$.
\end{proof}

%From now on fix

\begin{lemma}
Let $\Phi \in \Sc(H^n(\R \oplus \R )) $ and $\Psi \in \Sc(H^n(\C))$ such that
%$ \Omega_{\R\oplus \R}^{n,\psi}(\Phi)=\gamma(g,\psi)\Omega_{\C}^{n,\psi}(\Psi)$.
$\Omega_{\R\oplus \R}^{n,\psi}(\Phi)=\gamma\Omega_{\C}^{n,\psi}(\Psi)$.
 Let $g=\begin{pmatrix}
       x & 0\\
        0 & y
      \end{pmatrix},$ where $x \in S^{i}(D)$ and $y \in H^{n-i}(D)$. Then $\Omega_{\R\oplus \R}^{n,\psi}(\Phi,g)= \gamma(g,\psi) \Omega_{\C}^{n,\psi}(\Psi,g).$
\end{lemma}
This lemma follows from the induction hypotheses using intermediate Kloostermann integrals, i.e. integration over $N_i^n(D)$ (cf. \S\S \ref{IntKloosInt}).

\begin{lemma}
Let $\Phi \in \Sc(H^n(\R \oplus \R )) $ and $\Psi \in \Sc(H^n(\C))$ such that $\Omega_{\R\oplus \R}^{n,\psi}(\Phi)=\gamma\Omega_{\C}^{n,\psi}(\Psi)$.
Let $g=aw_n$ where $a \in \R^{\times}$. Then $\Omega_{\R\oplus \R}^{n,\psi}(\Phi,g)=\gamma(g,\psi) \Omega_{\C}^{n,\psi}(\Psi,g).$
\end{lemma}
This lemma follows from the previous one
using Lemma \ref{lem:SimpInv}.

The theorem follows now from the last 3 Lemmas. %??and Lemma \ref{lem:RelForm}.

\appendix
\section{Schwartz functions on Nash manifolds} \label{AppSchwartz}
\setcounter{lemma}{0}

In this appendix we give some complementary facts about Nash
manifolds and Schwartz functions on them and prove Property \ref{PropDixMal} and Theorems
 \ref{thm:NegCoinvPrel} and \ref{DualUncerPrel} from the
preliminaries.

\begin{theorem}[Local triviality of Nash manifolds]\label{LocTriv}
\label{loctriv} Any Nash manifold can be covered by finite number
of open submanifolds Nash diffeomorphic to $\R^n$.
\end{theorem}
For proof see \cite[Theorem I.5.12]{Shi}.

\begin{theorem} \label{NashTube}[Nash tubular neighborhood]
Let $M$ be a Nash manifold and $Z \subset M$ be closed Nash
submanifold. Then there exists an finite cover $Z = \cup Z_i$ by open Nash submanifolds of $Z$,
and open embeddings $N_{Z_i}^M \hookrightarrow M$ that are
identical on the zero section.
\end{theorem}

This follows from e.g. \cite[Corollary 3.6.3]{AG_Sc}.

\begin{notation}
We fix a system of semi-norms on $\Sc(\R^n)$ in the following way:
$$\gN_{k}(f):= \max_{\{\alpha \in \Z_{\geq 0}^n\, | \, |\alpha| \leq k\}} \max_{\{\beta \in \Z_{\geq 0}^n\, | \, |\beta| \leq k\}} \sup_{x \in \R^n}| x^{\alpha} \frac{\partial^{|\beta|}}{(\partial x)^\beta} f|.$$
\end{notation}

\begin{notation}
For any Nash vector bundle $E$ over $X$ we denote by $\Sc(X,E)$
the space of Schwartz sections of $E$.
\end{notation}

The properties of Schwartz functions on Nash manifolds listed in
the preliminaries hold also for Schwartz sections of Nash bundles.

\begin{remark}\label{rem:push}
One can put the notion of push of Schwartz functions in a more invariant setting.
Let $\phi:X \to Y$ be a morphism of Nash manifolds. Let $E$ be a bundle on $Y$. Let $E'$ be a bundle on $X$ defined by $E':=\phi^*(E \otimes D_Y^{-1})\otimes D_X$, where $D_X$ and $D_Y$ denote the bundles of densities on $X$ and $Y$. Then we have a well defined map $\phi_*:\Sc(X,E') \to \Sc(Y,E)$.
\end{remark}

\subsection{Analog of Dixmier-Malliavin theorem}$ $

In this subsection we prove Property \ref{PropDixMal}. Let us remind its formulation.
%Rami
\begin{thm}\label{ThmDixMal}
Let $\phi:M \to N$ be a Nash map of Nash manifolds. Then
multiplication defines an onto map $\Sc(M) \otimes \Sc(N)
\twoheadrightarrow \Sc(M)$.% This gives us an action of Schwartz
%functions on $M$ on Schwartz functions on $N$.
\end{thm}
First let us remind the formulation of the classical Dixmier-Malliavin theorem.
\begin{thm}[see \cite{DM}]
Let a Lie group $G$ acct continuously on a \Fre space $E$. Then $C^\infty_c(G) E=E^\infty,$ where $E^\infty$ is the subspace of smooth vectors in $E$ and $C^\infty_c(G)$ acts on $E$ by integrating the action of $G$.
\end{thm}
\begin{cor}\label{DMcor}
Let $L \subset V$ be finite dimensional linear spaces, and let $L$ act on $V$ by translations. Then $\Sc(L) * \Sc(V)=\Sc(V)$, where $*$ means convolution.
\end{cor}

\begin{proof}[Proof of Theorem \ref{ThmDixMal}]
Step 1. The case $N= \R^{n},\, M=\R^{n+k}, \, \phi$ is the projection.\\
Follows from Corollary \ref{DMcor} after applying Fourier transform.

Step 2. The case $N= \R^{n},\, M=\R^{k}, \, \phi$ -  general.\\
Identify $N$ with the graph of $\phi$ in $N \times M$. The assertion follows now from the previous step using Property \ref{Extension}.

Step 3. The general case.\\
Follows from the previous step using Property \ref{pCosheaf} and Theorem \ref{loctriv}.
\end{proof}

%Rami
\subsection{Coinvariants in Schwartz functions}\label{AppCoinv}$ $

%----------------------------------------------------------------------------------------------------------------------------------------------------------------------------------------------------
%----------------------------------------------------------------------------------------------------------------------------------------------------------------------------------------------------
%----------------------------------------------------------------------------------------------------------------------------------------------------------------------------------------------------
%----------------------------------------------------------------------------------------------------------------------------------------------------------------------------------------------------

\begin{defn}
Let a Nash group $G$ act on a Nash manifold $X$. A {\bf tempered $G$-equivariant bundle} $E$ over $X$ is a Nash bundle $E$ with an equivariant structure $\phi:a^*(E) \to p^*(E)$ (here $a:G\times X\to X$ is the action map and  $p:G\times X\to X$ is the projection) such that $\phi$
corresponds to a tempered section of the bundle $\Hom(a^*(E), p^*(E))$ (for the definition of tempered section see e.g. \cite{AG_Sc}), and for any element $\alpha$ in the Lie algebra of $G$ the derivation map $a(\alpha):C^\infty(X,E) \to C^\infty(X,E)$ preserves the sub-space of Nash
sections of $E$.
\end{defn}

In this subsection we prove the following generalization of Theorem \ref{thm:NegCoinvPrel}.
\begin{thm} \label{NegCoinv}
Let a connected algebraic group $G$ act on a real algebraic manifold $X$.
Let $Z$ be a $G$-invariant Zariski closed subset of $X$. Let $\g$ be the Lie algebra of $G$.
Let $E$ be a  tempered $G$-equivariant bundle over $X$.
Suppose that for any $z \in Z$ and $k \in \Z_{\geq 0}$ we have
$$(E|_z \otimes \Sym^k(CN_{z,Gz}^X) \otimes ((\Delta_G)|_{G_z}/\Delta_{G_z}))_{\g_z} =0.$$
Then $$(\Sc(X,E)/\Sc(X-Z,E))_\g=0.$$
\end{thm}
% This theorem implies Theorem \ref{NegCoinvPrel}  since  the action of $G_z$ on $\Sym^k(CN_{z,Gz}^X) \otimes ((\Delta_G)|_{G_z}/\Delta_{G_z})$ is algebraic and hence if $G$ is unipotent this action is unipotent and therefore if $(E|_z)_{\g_z} =0$ then $$(E|_z \otimes \Sym^k( CN_{z,Gz}^X) \otimes
% ((\Delta_G)|_{G_z}/\Delta_{G_z}))_{\g_z} =0.$$

For the proof of this theorem we will need some auxiliary results.
\begin{lem}\label{gr0all0}
Let $V$ be a representation of a Lie algebra $\g$. Let $F$ be a finite $\g$-invariant filtration of $V$. Suppose $gr_F(V)_\g=0$. Then $V_\g=0$.
\end{lem}
The proof is evident by induction on the length of the filtration.
\begin{lem}\label{gr0all0_inf}
Let $V$ be a representation of a finite dimensional Lie algebra $\g$. Let $F_i$ be a countable decreasing $\g$-invariant filtration of $V$. Suppose $\bigcap F^i(V)=0$, $F^0(V)=V$ and that the canonical map $V \to \lim \limits_\ot(V/F^i(V))$ is an isomorphism. Suppose also that $gr_F^i(V)_\g=0$. Then $V_\g=0$.
\end{lem}
This lemma is standard and we included its prove for the sake of completeness.
\begin{proof}
We have to prove that the map $\g \otimes V \to V$ is onto. Let $v \in V$. We will construct in an inductive way a sequence of vectors $w_i \in \g \otimes V/F^i(V)$ s.t. their image under the action map $\g \otimes V/F^i(V) \to V/F^i(V)$ coincides with the image of $v$ under the quotient map $V \to V/F^i(V)$. Define $w_0=0$. Suppose we have already defined $w_n$ and we have to define $w_{n+1}$. Let $w_{n+1}'$ be an arbitrary lifting of $w_{n}$ to $\g \otimes V/F^{n+1}(V).$ Let $v_{n+1}'$ be the image of $w_{n+1}'$
under the action map $\g \otimes V/F^{n+1}(V) \to V/F^{n+1}(V)$ and let $v_{n+1}$ be the image of
 $v$ under the quotient map $V \to V/F^{n+1}(V)$. Let $dv=v_{n+1}-v_{n+1}'$. Clearly $dv$ lies in $F^{n}(V)/F^{n+1}(V).$ Let $dw$ be its lifting to $\g \otimes (F^{n}(V)/F^{n+1}(V)).$ Denote $w_{n+1}=w_{n+1}'+dw.$

Since $\g$ is finite dimensional, the canonical map $\g \otimes V \to \lim\limits _\ot \g \otimes  (V/F^i(V))$ is an isomorphism. Therefore there exists a unique $w \in \g \otimes V$ s.t. its image in $\g \otimes  (V/F^i(V))$ is $w_i$. Thus the image of $w$ under the map $\g \otimes V \to V$ is $v$.
\end{proof}
\begin{notation}
Let $Z$ be a locally closed semi-algebraic subset of a Nash manifold $X$. Let $E$ be a Nash bundle over $X$. Denote $$\Sc_X(Z,E):=\Sc(X-(\overline{Z}-Z))/\Sc(X-\overline{Z},E).$$
\end{notation}

\begin{lem}\label{strat_fil}
Let $X$ be a Nash manifold and $Z \subset X$ be a locally closed semi-algebraic subset. Let $E$ be a Nash bundle over $X$. Let $Z_i$ be a finite stratification of $Z$ by locally closed semi-algebraic subsets. Then $\Sc_X(Z,E)$ has a canonical filtration s.t. $$gr_i(\Sc_X(Z,E)) \cong \Sc_X(Z_i,E)).$$
\end{lem}
\begin{proof}
It follows immediately from property \ref{pOpenSet}.
\end{proof}

\begin{lem}
Let X be a Nash manifold and $Z \subset X$ be Nash submanifold. Then $\Sc_X(Z)$ has a canonical countable decreasing filtration satisfying
%$\bigcap V^i=0$
$\bigcap (\Sc_X(Z))^i=0$
s.t. $gr_i(\Sc_X(Z,E)) \cong \Sc(Z,\Sym^i(CN_Z^X)\otimes E).$
\end{lem}
\begin{proof}
It follows from the proof of Corollary 5.5.4. in \cite{AG_Sc}.
\end{proof}

\begin{lem} [E. Borel] \label{Borel_lem}
Let X be a Nash manifold and $Z \subset X$ be Nash submanifold. Then
the natural map  $$\Sc_X(Z,E) \to \lim_\ot(\Sc_X(Z,E))/\Sc_X(Z,E))^i)$$
is an isomorpihsm.
\end{lem}
\begin{proof}
Step 1. Reduction to the case when $X$ is a total space of a bundle over $Z$.\\
It follows immediately from Theorem \ref{NashTube}.

Step 2. Reduction to the case when $Z=\R^n$ is standardly embedded inside $X=\R^{n+k}$.\\
It follows immediately from Theorem \ref{LocTriv} and Property \ref{pCosheaf}.

Step 3. Proof for the case when $Z=\R^n$ standardly embedded inside $X=\R^{n+k}$.\\
It is the same as the proof of the classical Borel Lemma.
\end{proof}

\begin{definition}
We call an action of a Nash group $G$ on a Nash manifold $X$ factorisable if the map $\phi_{G,X}:G \times X \to X \times X$ defined by $(g,x) \mapsto (gx,x)$ has a Nash image and is a submersion onto it.
\end{definition}

\begin{thm} [Chevalley] \label{strat_fac}
Let a real algebraic group act on a real algebraic variety $X$. Then there exists a finite $G$-invariant smooth stratification $X_i$ of $X$ s.t. the action of $G$ on $X_i$ is factorisable.
\end{thm}
\begin{proof}
By the classical  Chevalley Theorem there exists a Zariski open subset $ U \subset G \times X$ s.t. the map $\phi_{G,X}|_{U}$ is a submersion to its smooth image. Let $X_0 \subset X$ be the projection of $U$ to $X$. It is easy to see that $\phi_{G,X}|_{ G \times X_0}$  is a
submersion to its smooth image. The theorem now follows by Noetherian induction.
\end{proof}

\begin{thm} \label{fac_co_inv}
Let a Nash group $G$ act factorisably on a Nash manifold $X$ and $E$ be a tempered $G$-equivariant bundle over $X$. Suppose that for any $x \in X$ we have $$((E|_x \otimes ((\Delta_G)|_{G_x}/\Delta_{G_x})))_{\g_x} =0.$$  Then $$(\Sc(X,E))_\g=0.$$
\end{thm}
For the proof see section \ref{pf_fac_co_inv} below.

Now we ready to prove Theorem \ref{NegCoinv}.
\begin{proof}[Proof of Theorem \ref{NegCoinv}]$ $\\
Step 1. Reduction to the case that the action of $G$ on $Z$ is factorisable.\\
It follows from Theorem \ref{strat_fac} and Lemmas \ref{strat_fil} and \ref{gr0all0}.

Step 2. Reduction to the case that the action of $G$ on $Z$ is factorisable and $Z=X$.\\
It follows from Theorems \ref{Borel_lem} and \ref{gr0all0_inf}.

Step 3. Proof for the case that the action of $G$ on $Z$ is factorisable and $Z=X$.\\
It follows from Theorem \ref{fac_co_inv}.
\end{proof}
\subsubsection{Proof of Theorem \ref{fac_co_inv}}\label{pf_fac_co_inv}$ $
\begin{notation}
Let $\phi:X \to Y$ be a map of (Nash) manifolds.\\
(i) Denote $D_Y^X:=D_\phi :=\phi^*(D_Y^*)\otimes D_X$.\\
(ii) Let $E \to Y$ be a (Nash) bundle. Denote $\phi^?(E)=\phi^*(E)\otimes D_Y^X$.
\end{notation}
\begin{rem}
Note that\\
(i)  If $\phi$ is a submersion then for all $y \in Y$ we have $D_Y^X|_{\phi^{-1}(y)} \cong D_{\phi^{-1}(y)}$.\\
(ii) If $\phi$ is a submersion then by Remark \ref{rem:push} we have a well defined map $\phi_*: \Sc^*(X,\phi^?(E)) \to \Sc^*(Y,E).$\\
(iii) If a Lie group $G$ acts on a smooth manifold $X$ and $E$ is a $G$-equivariant vector bundle (i.e. we have a map $p^*(E) \to a^*(E)$, where $p:G \times X \to X$ is the projection and $a:G \times X \to X$ is the action) then we also have a natural map $p^?(E) \to a^?(E).$ If $G$, $X$ and $E$ are Nash and the actions of $G$ on $X$ and $E$ are Nash then the map $p^?(E) \to a^?(E)$ is Nash. If the action of $G$ on $E$ is tempered then the  map $p^?(E) \to a^?(E)$ corresponds to a tempered section of  $\Hom(p^?(E),a^?(E))$.
\end{rem}

\begin{notation}
Let $G$ be a Nash group. We denote $$\Sc(G,D_G)_0:=\{ f \in \Sc(G,D_G) | \int_G f=0\}.$$
\end{notation}

\begin{lem}\label{g_coinv_simp}
Let $G$ be a connected Nash group and $\g$ be its Lie algebra. Then
$\g \Sc(G,D_G) = \Sc(G,D_G)_0.$
\end{lem}
\begin{proof}
The inclusion
$\g \Sc(G,D_G) \subset \Sc(G,D_G)_0$ is evident.
The theorem follows now from the fact that $\dim \Sc(G,D_G)_{\g}=1$, which is proved in the same way as Proposition 4.0.11 in \cite{AG_RhamShap}.
\end{proof}

\begin{notation}
Let $G$ be a Nash group, $X$ be a Nash manifold and $E$ be a Nash bundle over $X$. Let $p:G\times X \to X$ be the projection. Denote by $\Sc(G \times X,p^?(E))_{0,X}$ the kernel of the map $p_*:\Sc(G \times X,p^?(E)) \to \Sc(X,E)$. In cases when there is no ambiguity we will denote it just by $\Sc(G \times X,p^?(E))_{0}.$
\end{notation}
\begin{lem}
Let $G$ be a Nash group, $X$ be a Nash manifold and $E$ be a Nash bundle over $X$. Let $p:G\times X \to X$ be the projection.  Then $$\Sc(G \times X,p^?(E))_0 \cong \Sc(G,D_G)_0 \ctp \Sc (X,E).$$
\end{lem}
\begin{proof}
The sequence $$0\to \Sc(G,D_G)_0 \to \Sc(G,D_G) \to \C \to 0$$ is exact. Therefore by Proposition \ref{prop:ctp_ext}
the  sequence $$0\to \Sc(G,D_G)_0 \ctp \Sc(X,E) \to \Sc(G,D_G) \ctp \Sc(X,E) \to \Sc(X,E) \to 0$$
is also exact.
 Thus it is enough to show that the map $\Sc(G,D_G) \ctp \Sc(X,E) \to \Sc(X,E)$ corresponds to the map $p_*:\Sc(G \times X,p^?(E)) \to \Sc(X,E)$ under the identification $\Sc(G,D_G) \ctp \Sc(X,E) \cong \Sc(G \times X,p^?(E))$. Since those maps are continuous it is enough to check that they are the same on the image of $\Sc(G,D_G) \otimes \Sc(X,E)$, which is
evident.
\end{proof}

\begin{cor}
Let $G$ be a Nash group, $X$ be a Nash manifold and $E$ be a Nash bundle over $X$. Let $\g$ be the Lie algebra of $G$. Let $p:G\times X \to X$ be the projection. Let $G$ act on $\Sc(G \times X,p^?(E))$ by acting on the $G$ coordinate. Then
$\g \Sc(G \times X,p^?(E)) = \Sc(G \times X,p^?(E))_0.$
\end{cor}
\begin{proof}
The corollary follows from the last two lemmas using Proposition \ref{prop:ctp_ext}.
\end{proof}

\begin{cor}\label{coinv_g}
Let $G$ be a connected Nash group and $\g$ be its Lie algebra. Let $G$ act on a Nash manifold $X$ and let $E$ be a tempered $G$-equivariant bundle over $X$. Let $p:G\times X \to X$ be the projection. Let $a:G\times X \to X$ be the action map.

Then $\g \Sc(X,E)$ is the image $a_*(\Sc(G \times X,a^?(E))_0)$ where $\Sc(G \times X,a^? (E))_0$ denotes the  image of $\Sc(G \times X, p^?(E))_0$ under the identification  $\Sc(G \times X,p^?(E)) \cong  \Sc(G \times X,a^?(E))$.
\end{cor}
\begin{proof}
Let $G$ act on $\Sc(G \times X,p^?(E))$ by acting on the $G$ coordinate.  The identification  $\Sc(G \times X,p^?(E))
 \cong  \Sc(G \times X,a^?(E))$  gives us an action of $G$ on $\Sc(G \times X,a^?(E))$. It is easy to see that  $a_*:\Sc(G \times X,a^?(E)) \to \Sc(X, E)$ is a morphism of
$G$-representations. By property  \ref{NashSub} $a_*$ is surjective. Therefore $(\g \Sc(X,E)) = a_*((\g \Sc(G \times X,a^?(E))).$  The assertion follows now by the previous corollary.
\end{proof}
\begin{definition}
A \textbf{Nash family of groups} over a Nash manifold $X$ is a surjective submersion $G \to X$, a Nash map $m: G \times_X G\to G$ and a Nash section $e:X \to G$ s.t. for any $x \in X$ the map $m|_{G|_x \times G|_x}$ gives a group structure on the fiber $G|_x$ and $e(x)$ is the unit of
this group.
\end{definition}
\begin{definition}
A \textbf{Nash family of Lie algebras} over a Nash manifold $X$ is a Nash bundle $\g \to X$, a Nash section $m$ of the bundle $\Hom(\g \otimes \g,\g)$ s.t. for any $x \in X$ the map $m(x): \g|_x \otimes \g|_x \to \g|_x$ gives a Lie algebra structure on the fiber $\g|_x$ .
\end{definition}
\begin{definition}
A Nash family of Lie algebras of a Nash family of groups  $G$ over a Nash Manifold $X$ is the bundle $e^{*}(N_{e(X)}^G)$ equipped with the natural structure of a Nash family of Lie algebras. We will denote it by $\Lie(G)$.
\end{definition}

\begin{notation}
Let $G$ be a Nash family of groups over a Nash manifold $X$. Let $E$ be  a bundle over $X$. Let $p:G \to X$ be the projection. Denote by $\Sc(G ,p^?(E) )_{0,G}$ the kernel of the map $p_*:\Sc(G,p^?(E))) \to \Sc(X,E)$. If there is no ambiguity we will denote it by $\Sc(G ,p^?(E) )_{0}$.
\end{notation}

\begin{lem}\label{fem_coinv_inc}
Let $G$ be a Nash family of groups over a Nash manifold $X$ and $\g$ be its family of Lie algebras. Then  the image of the natural map $\Sc(X,\g) \otimes \Sc(G,D_G) \to \Sc(G,D_G)$ is included in $\Sc(G,D_G)_0$.
\end{lem}
\begin{proof}
It follows immediately from the case when $X$ is one point and $E$ is $\C$ which follows from Lemma \ref{g_coinv_simp}.
\end{proof}
\begin{definition}
A Nash family of representations of a Nash family of Lie algebras $\g$ over a Nash manifold $X$ is a bundle $E$ over $X$ and a Nash section $a$ of the bundle $\Hom(\g \otimes E, E)$ s.t.  for any $x \in X$  the map $a(x): \g|_x \otimes E|_x \to E|_x$ gives a structure of a representation
of $\g|_x$ on the fiber $E|_x$.
\end{definition}
\begin{definition}
Let $G$ be a Nash family of groups over a Nash manifold $X$. Let $\g$ be its family of Lie algebras. Let $p:G \to X$  be the projection.
A tempered (finite dimensional) family of representations of $G$ is a pair $(E,a)$  where $E$ is a bundle over $X$ and $a$ is a tempered section of the bundle $\End(p^*E)$ s.t.  for any $x \in X$  the section $a|_{G|_x}$  gives a structure of a representation of $G|_x$ on
the fiber $E|_x$ and s.t. the differential of $a$ considered as a section of $\Hom(\g \otimes E, E)$ gives a structure of a Nash family of representations of $\g$ on $E$.
\end{definition}
Lemma \ref{fem_coinv_inc} gives us the following corollary.
\begin{cor}
Let $G$ be a Nash family of groups over a Nash manifold $X$ and $\g$ be its Lie algebra. Let $(E,a)$ be a tempered (finite dimensional) family of representations of $G$. Let $\phi$ denote the composition
$ \Sc(G,p^?(E)) \overset{a}{\to} \Sc(G,p^?(E)) \overset{p_*}{\to} \Sc(X,E)$. Then  the image of the natural map $\Sc(X,\g) \otimes \Sc(X,E) \to \Sc(X,E)$ is included in $\phi(\Sc(G,p*(E)\otimes D_G)|_0)$.
\end{cor}
\begin{lem}
Let $\g$ be a Nash family of Lie algebras over a Nash manifold $X$. Let $E$ be a Nash family of its representations. Consider $\Sc(X,\g)$ as a Lie algebra and $\Sc(X,E)$ as its representation. Suppose that for any $x \in X$ we have  $(E|_x)_{\g|_x}=0.$  Then $(\Sc(X,E))_{\Sc(X,\g)}=0.$
\end{lem}
\begin{proof}
For any $x\in X$ denote by $a_x$ the map ${\g|_x} \otimes E|_x \to E|_x.$
By property \ref{pCosheaf} we may  assume that $E$ and $\g$ are trivial bundles with fibers $V$ and $W$. Fix a basis for $V$ and $W$ and the corresponding basis for $W \otimes V$. Let $\fS$ be the collection of coordinate subspaces of $W \otimes V$ of dimension $\dim V$. For any $L \in \fS$
denote $U_L=\{x \in X| a_x(L)=V\}$. Clearly $X=\bigcup U_L$. Thus by property \ref{pCosheaf} we may assume that $X=U_L$ for some $L$. For this case the lemma is evident.
\end{proof}
\begin{cor} \label{g_fam_coinv_0}
Let $G$ be a Nash family of groups over a Nash manifold $X$ and $\g$ be its Lie algebra. Let $(E,a)$ be a tempered (finite dimensional) family of representations of $G$. Let $\phi $ denote the composition
$$ \Sc(G,p^?(E)) \overset{a}{\to} \Sc(G,p^?(E)) \overset{p_*}{\to} \Sc(X,E).$$ Suppose that for any $x \in X$ we have $(E|_x)_{\g_x}=0.$
Then   $$\phi(\Sc(G,p^?(E))|_0)=\Sc(X,E).$$
\end{cor}
\begin{definition}
We call a set $G$ equipped with a map $m:G\times G\times G \to G$ \textbf{ a torsor} if there exists a group structure on $G$ s.t. $m(x,y,z)=z((z^{-1}x) (z^{-1}y))$. One may say that a torsor is a group without choice of identity element.
\end{definition}
\begin{definition}
A Nash family of torsors over a Nash manifold $X$ is a surjective submersion $G \to X$ and a Nash map $m: G \times_X G\times_X G\to G$ s.t. for any $x \in X$  the map $m|_{G|_x \times G|_x \times G|_x}$ gives a torsor structure on the fiber $G|_x$.
\end{definition}
\begin{definition}
Let $G$ be a Nash family of torsors over a Nash manifold $X$. Let $p:G \to X$  be the projection. Consider $\Ker dp$ as a subbundle of $T G$. It has a natural structure of a family of Lie algebras over $G$. We will call this family the family of Lie algebras of $G$.
\end{definition}

\begin{rem}
One could define the family of Lie algebras of $G$ to be a family of Lie algebras over $X$. This definition would be more adequate, but it is technically harder to phrase it. We did not do it since it is unnecessary for our purposes.
\end{rem}

\begin{definition}
A {\bf representation of a  torsor $G$} is a pair $(V,W)$ of vector spaces and a morphism of torsors $G \to \Iso(V,W)$.
\end{definition}

\begin{definition}
Let $G$ be a Nash family of torsors over a Nash manifold $X$. Let  $\g$ be its family of Lie algebras. Let $p:G \to X$  by the projection.
A {\bf tempered (finite dimensional) family of  representations of $G$} is a triple $(E,L,a)$, where $E$ and $L$ are (Nash) bundles over $X$ and $a$ is a tempered section of the bundle $Hom(p^*E,p^*L)$ s.t.  for any $x \in X$  the section $a|_{G|_x}$  gives a structure of a representation of $G|_x$ on
the fibers $E|_x$ and $L|_x$  and s.t. the differential of $a$ considered as a section of $\Hom(\g \otimes p^*L, p^*L)$ gives a structure of a Nash family of representations of $\g$ on $p^*L$.
\end{definition}

Corollary \ref{g_fam_coinv_0} gives us the following corollary.
\begin{cor}\label{tor_fam_coinv_0}
Let $G$ be a Nash family of torsors over a Nash manifold $X$ and $\g$ be its Lie algebra. Let $(E,L,a)$ be a tempered (finite dimensional) family of representations of $G$. Let $\phi $ denote the composition
$$ \Sc(G,p^?(E)) \overset{a}{\to} \Sc(G,p^?(L)) \overset{p_*}{\to} \Sc(X,L).$$  Suppose that for any
 %$x \in X$
$x \in G$
  we have
% $(L|_x)_{\g_x}=0.$
 $(L|_{p(x)})_{\g_x}=0.$
Then $$\phi(\Sc(G,p^?(E))_0)=\Sc(X,L).$$
\end{cor}
\begin{proof}
It follows from Corollary \ref{g_fam_coinv_0} using property \ref{pCosheaf} and \cite[Theorem 2.4.16]{AG_RhamShap}.
%\Dima{?? This is the arxiv reference!}
\end{proof}

Now we are ready to prove Theorem \ref{fac_co_inv}.
\begin{proof}[Proof of Theorem  \ref{fac_co_inv}]
By Lemma \ref{coinv_g} it is enough to show that $$a_*(\Sc(G \times X, a^?(E))_0)=\Sc(X,E).$$
Let $Y$ be the image of the map $b:G \times X \to X \times X$ defined by $b(g,x)=(x,gx)$. Let $E_i=p_i^?(E)$ for $i=1,2$. Here $p_i:Y \to X$ is the projection
to the $i$'s coordinate. Note that $G \times X$ has a natural structure of a family of torsors over $Y$ and the $G$-equivariant structure on $E$  gives a family of representations
$(\psi,E_1,E_2)$ of the family of torsors $b:G \times X \to Y$. It is enough to show that $$b_*(\Sc(G \times X, a^?(E)) )_0)=\Sc(Y,E_2).$$ Recall that
$\Sc(G \times X, a^?(E))_0$ is the image of $\Sc(G \times X, p^?(E))_{0,X}$ under the
identification $\phi:\Sc(G \times X, a^?(E)) \to  \Sc(G \times X, p^?(E)))$. Note that
$\Sc(G \times X, p^?(E))_{0,X}$ includes $\Sc(G \times X, p^?(E))_{0,Y}$. Therefore it is enough to show that the image of $\Sc(G \times X, b^?(E_1))_{0,Y}$ under the composition $$\Sc(G \times X, b^?(E_1)) \to \Sc(G \times X, b^?(E_2)) \to \Sc(Y, E_2)$$ is $\Sc(Y,E_2).$ This follows by Corollary \ref{tor_fam_coinv_0} from the fact that for every $y \in Y$ we have $((E_2)|_y)_{\g_{p_2(y)}}=0.$ This fact is a reformulation of the fact that $$((E|_{p_2(y)} \otimes ((\Delta_G)|_{G_{p_2(y)}}/\Delta_{G_{p_2(y)}})))_{\g_{p_2(y)}} =0,$$ which is part of the assumptions of the theorem.
\end{proof}

%--------------------------------------------------------------------------------------------------------------------------------------------------------------------------------------------------
%----------------------------------------------------------------------------------------------------------------------------------------------------------------------------------------------------
%----------------------------------------------------------------------------------------------------------------------------------------------------------------------------------------------------
%----------------------------------------------------------------------------------------------------------------------------------------------------------------------------------------------------

\subsection{Dual uncertainty principle}\label{AppDualUn}

%Fourier Transform}

%\begin{theorem} \label{DualUncer}
%Let $V$ be a linear space, $L \subset V$ and $L' \subset V^*$ be
% subspaces. Suppose that $(L')_\bot \nsubseteqq L$. Let $E:= \{f \in
% \Sc(V)| \, f \text{ vanishes with all its derivatives on } L\}$ and $E':= \{f \in
% \Sc(V^*)| \, f \text{ vanishes with all its derivatives on } L'\}$.
%Then $$E + \widehat{E'}=\Sc(V).$$
%\end{theorem}

\begin{notation}
Let $V$ be a finite dimensional real vector space. Let $\psi$ be a
non-trivial additive character of $\R$. Let $\mu $ be a Haar
measure on $V$. Let $f \in \Sc(V)$ be a function. We denote by
$\widehat{f} \in \Sc(V^*)$ the Fourier transform of $f$ defined by
$\mu$ and $\psi$.
\end{notation}

In this subsection we prove the following generalization of Theorem \ref{DualUncerPrel}.

\begin{theorem} \label{DualUncer}
Let $V$ be a linear space, $L \subset V$ and $L' \subset V^*$ be
 subspaces. Suppose that $(L')_\bot \nsubseteqq L$.
 %Let $E:= \{f \in
% \Sc(V)| \, f \text{ vanishes with all its derivatives on } L\}$ and $E':= \{f \in
% \Sc(V^*)| \, f \text{ vanishes with all its derivatives on } L'\}$.
Then $$\Sc(V-L) + \widehat{\Sc(V^*-L')}=\Sc(V).$$
\end{theorem}

The following lemma is obvious.
\begin{lem}
There exists $f \in \Sc(\R)$ such that $f$ vanishes at 0 with all
its derivatives and $\Fou(f)(0)=1$.
\end{lem}

\begin{cor}
Let $L$ be a quadratic space. Let $V:=L \oplus \R$ be enhanced
with the obvious quadratic form. Let $g \in \Sc(L)$. Then there
exists $f \in \Sc(V)$ such that $f \in \Sc(V-L)$ and $\Fou(f)|_L=g$.
\end{cor}

\begin{cor}
Let $L$ be a quadratic space. Let $V:=L \oplus \R e$ be enhanced
with the obvious quadratic form. Let $g \in \Sc(L)$. Let
$i$ be a natural number. Then there exists $f \in \Sc(V)$ such that $f \in \Sc(V-L)$,
$\frac{\partial^i \Fou(f)}{(\partial e)^i}|_L=g$ and
$\frac{\partial^i \Fou(f)}{(\partial e)^j}|_L=0$ for any $j<i$.
\end{cor}

\begin{cor}
Let $L$ be a quadratic space. Let $V:=L \oplus \R e$ be enhanced
with the obvious quadratic form. Let $g \in \Sc(L)$. Then for all $i$ and $\eps$ there exists $f \in \Sc(V)$ such that
$\gN_{i-1}(f) < \eps$, %(we have to denote \gN_{i}(f) in the preliminaries)
$f \in \Sc(V-L)$,  $\frac{\partial^i
\Fou(f)}{(\partial e)^i}|_L=g$ and $\frac{\partial^i
\Fou(f)}{(\partial e)^j}|_L=0$ for any $j<i$.
\end{cor}
\begin{proof}
Let $f \in \Sc(V)$ be s.t. $f \in \Sc(V-L)$, $\frac{\partial^i
\Fou(f)}{(\partial e)^i}|_L=g$ and $\frac{\partial^i
\Fou(f)}{(\partial e)^j}|_L=0$ for any $j<i$.

Let $f^t \in \Sc(V)$ defined by $f^t(x+\alpha e_1)=t^{i+2} f(x+t
\alpha e_1)$. It is easy to see that $\frac{\partial^i
\Fou(f)}{(\partial e)^i}|_L=g$ and $\frac{\partial^i
\Fou(f)}{(\partial e)^j}|_L=0$ for any $j<i$. Also it is easy to
see that $\lim \limits _{t\to 0} \gN_{i-1}(f^t)=0.$ This implies the assertion.
\end{proof}

\begin{cor}
Let $L$ be a quadratic space. Let $V:=L \oplus \R e$ be enhanced
with the obvious quadratic form. Let $\{g_i\}_{i=0}^{\infty} \in
\Sc(L)$. Then there exists $f \in \Sc(V)$ such that $f$ vanishes
on $L$ with all its derivatives and $\frac{\partial^i
\Fou(f)}{(\partial e)^i}|_L=g_i$.
\end{cor}

\begin{proof}
Define 3 sequences of functions $f_i,h_i \in S(V),g_i' \in S(L)$
recursively in the following way: $f_0=0$.
$g_i'=g_i-\frac{\partial^i \Fou(f_{i-1})}{(\partial e)^i}|_L$. Let
$h_i\in \Sc(V)$ s.t. $h_i \in \Sc(V-L)$, $\gN_{i-1}(h_i) < 1/2^i$,
$\frac{\partial^i \Fou(h_i)}{(\partial e)^i}|_L=g_i'$ and
$\frac{\partial^i \Fou(h_i)}{(\partial e)^j}|_L=0$ for any $j<i$.
Define $f_i=f_{i-1}+h_i$.

Clearly $f:=\lim \limits _{i \to \infty} f_i$ exists and satisfies the
requirements.
\end{proof}

\begin{cor}
Let $L$ be a quadratic space. Let $V:=L \oplus \R e$ be enhanced
with the obvious quadratic form.

%Let $E:=\{ f\in \Sc(V)| \, f
%\text{ vanishes on }L \text{ with all its derivatives}\}$.

Then $\Sc(V-L) + \Fou(\Sc(V-L)) = \Sc(V).$
\end{cor}

\begin{proof}
Let $f \in \Sc(V)$. Let $f' \in  \Sc(V-L)$ s.t. $\frac{\partial^i
\Fou(f')}{(\partial e)^i}|_L=\frac{\partial^i \Fou(f)}{(\partial
e)^i}|_L$. Let $f''= f-f'$. Clearly $f'' \in  \Fou(\Sc(V-L))$.
\end{proof}

\begin{cor}
Let $V$ be a linear space, $L \subset V$ and $L' \subset V^*$ be
 subspaces of codimension $1$. Suppose that $(L')_\bot \nsubseteqq L$.
 %Let $E:= \{f \in
% \Sc(V)| \, f \text{ vanishes with all its derivatives on } L\}$ and $E':= \{f \in
% \Sc(V^*)| \, f \text{ vanishes with all its derivatives on } L'\}$.
Then $$\Sc(V-L) + \widehat{\Sc(V^*-L')}=\Sc(V).$$
\end{cor}

\begin{proof}
Choose a non-degenerate quadratic form on $V$ s.t. $L \bot
(L')_\bot$. This form gives an identification $V \to V^*$ which
maps $L$ to $L'$. Now the corollary follows from the previous
corollary.
\end{proof}

Now we are ready to prove Theorem \ref{DualUncer}.

\begin{proof}[Proof of Theorem \ref{DualUncer}]
Let $M \supset L$ be a sub-space in $V$ of  codimension $1$ s.t.
$M^\bot \nsubseteqq L'$. Let $M' \supset L'$ be a sub-space in
$V^*$ of  codimension $1$ s.t. $M^\bot \nsubseteqq M'$. The
theorem follows now from the previous corollary.
\end{proof}

%\section{Nuclear \Fre spaces} \label{NucFre}

\end{document}